\documentclass{article}
\usepackage[table]{xcolor}
\usepackage{amsfonts}
\usepackage{amsmath}
\usepackage{amssymb}
\usepackage{amsthm}
\usepackage{graphicx}
\usepackage[hidelinks]{hyperref}
\usepackage{array}
\usepackage{geometry}
\usepackage{tikz}
\usetikzlibrary{positioning}
\usetikzlibrary{arrows.meta}
\usetikzlibrary{shapes.geometric}
\usetikzlibrary{decorations.pathmorphing}
\usetikzlibrary{decorations.markings}
\usepackage{enumitem}
\usepackage{subcaption}

\newtheorem{prop}{Proposition}
\newtheorem{lem}{Lemma}
\newtheorem{coro}{Corollary}
\newtheorem{theo}{Theorem}

\theoremstyle{definition}
\newtheorem{defix}{Definition}
\newenvironment{defi}
  {\pushQED{\qed}\defix}
  {\popQED\enddefix}
\newtheorem{exax}{Example}
\newenvironment{exa}
  {\pushQED{\qed}\exax}
  {\popQED\endexax}

\newtheorem{remx}{Remark}
\newenvironment{rem}
  {\pushQED{\qed}\remx}
  {\popQED\endremx}

  \title{Invariant Sphere Theorem and Ring-Coupled Systems}
\author{ Pedro Soares\\
Instituto Superior de Economia e Gest\~{a}o, Universidade de Lisboa, Portugal\\
psoares@iseg.ulisboa.pt\footnote{This work was supported by FCT through the project CEMAPRE/REM
UID/06522/2025.}}
\date{July 2026} 
\begin{document}

\maketitle

\abstract{In this work, we show how heteroclinic networks can arise in a simple class of network dynamical systems through the application of the Invariant Sphere Theorem. Ring-coupled systems are ODE networks in
$\mathbb{R}^n$ in which each variable $x_i$ interacts only with its predecessor $x_{i-1}$.
We derive conditions on the coefficients of a cubic polynomial that guarantee the existence of a globally attracting invariant sphere via the Invariant Sphere Theorem.
Moreover, setting one of these coefficients to zero, we identify conditions on the remaining coefficients that guarantee the existence of a heteroclinic network on the invariant sphere.
Focusing on the case $n=3$, we investigate perturbations of the vanishing coefficient in a neighbourhood of the heteroclinic network. Under such perturbations, the heteroclinic network is destroyed and periodic orbits emerge. When the original heteroclinic network is asymptotically stable, the resulting periodic orbits shadow the network. In one parameter regime, a unique attracting periodic orbit appears and shadows the entire heteroclinic network. By contrast, when the heteroclinic network is not stable, repelling periodic orbits arise that shadow only part of the heteroclinic structure, namely half of the heteroclinic connections.
Symmetry plays a fundamental role throughout the analysis. Exploiting the symmetries of the system, we reduce the heteroclinic network to two homoclinic orbits. Furthermore, the local dynamics near the heteroclinic network can be studied on a quotient space consisting of an annulus attached to a Möbius band, with each homoclinic orbit lying on one of these components. This reduction provides a geometric framework for understanding the bifurcations and the emergence of the periodic dynamics.
}

\section{Introduction}

Heteroclinic networks arise in dynamical systems possessing symmetry or constrained geometries.
Classical examples include evolutionary game dynamics on simplexes, where flow-invariant faces support sequences of saddle equilibria connected by heteroclinic orbits \cite{HS98, PR23}, and equivariant systems, where invariant subspaces and an invariant sphere organise global dynamics to create robust heteroclinic networks \cite{GH88, F07}.
A distinct setting in which heteroclinic behaviour emerges is that of coupled cell systems: network structured dynamical systems whose architecture constrains the admissible vector fields and enforces suitable invariant subspaces where the heteroclinic connections lie \cite{F15,BG24,GS26}.

In this work, we will explore the emergence of heteroclinic networks in ring-coupled systems.
The ring architecture is one of the simplest network structures, consisting of cells arranged in a circular topology with unidirectional coupling, Figure~\ref{fig:ringnetwork}.
In this case, the network structure does not force the existence of suitable invariant subspaces for robust heteroclinic networks. The heteroclinic networks that we will find are not robust under perturbations of the coupled system.
First, we will apply the invariant sphere theorem to ring-coupled systems to obtain an invariant sphere.
Second, we find a heteroclinic cycle of the same size as the ring by adding a symmetric condition that forms a heteroclinic network between the poles of the sphere.
Third, we perturb the $3$-ring-coupled system to study its bifurcations, including the destruction of the heteroclinic network.

First, we study how the invariant sphere theorem can be applied to ring-coupled systems.
The invariant sphere theorem refers to homogenous polynomials with odd degree.
So we will focus on homogeneous cubic polynomials respecting the ring structure.
This means that we will study coupled cell systems given by a cubic homogeneous polynomial
$$q(y,z)=a y^3+b y^2 z + c y z^2+d z^3,$$ since each cell receives only one input, the variable $y$ corresponds to the internal state of the cell and $z$ to the state of its input cell. 
The $n$-ring-coupled system given $q(x,y)$ is 
$$Q^N:\mathbb{R}^n\rightarrow \mathbb{R}^n$$
$$(Q^N)_i(x)=q(x_{i},x_{i+1})$$
for $i=1,\dots,n$, $x=(x_1,\dots,x_n)$ and $x_{n+1}:=x_1$.
And we will study the dynamics of the system
\begin{equation}\label{eq:edoring}
\dot{x}=\lambda x + Q^N(x),
\end{equation}
where $x=(x_1,x_2,\dots,x_n)\in\mathbb{R}^n$ and $\lambda>0$. 
These systems respect the $\mathbb{Z}_2\times\mathbb{Z}_n$ symmetry given by the actions $x\mapsto -x$ and the cyclic permutation of the coordinates. 
However, the class of ring-coupled systems is smaller than the class of $\mathbb{Z}_2\times\mathbb{Z}_n$ equivariant systems.
We will provide algebraic conditions on the coefficients $a,b,c,d$ that guarantee the existence of an attracting invariant sphere for $\lambda>0$ using the invariant sphere theorem.
Moreover, the invariant sphere theorem also states that the dynamic on this invariant sphere is topologically conjugated to the  dynamic on the unitary sphere induced by the following phase vector field $$\dot{u}= Q^N(u)-\langle Q^N(u),u\rangle u,$$
where $u\in S^{n-1}$. 

Second, we study the existence of heteroclinic cycles on the invariant sphere.
We impose the condition $d=0$ to guarantee the existence of invariant subspaces that support the heteroclinic connections.
When $n \geq 3$, we give conditions on the coefficients $b,c$ ($a$ may be assumed to be $-1$ by a time-scaling) for the existence of heteroclinic cycles connecting the axial equilibria $\pm e_i$, $i=1,\dots,n$.
These conditions correspond to steady-state bifurcation conditions, transcritical or fold, at the axial equilibria $\pm e_i$ or a middle point between axial equilibria, respectively.
This forms a heteroclinic network connecting the poles of the sphere.
We also see that this heteroclinic network is not stable if $n>3$. And we give conditions on the coefficients $b,c$ for the stability of the heteroclinic network when $n=3$.

Finally, we focus on the case $n=3$ and study perturbation of the previous dynamic on the sphere.
For simplicity, we assume that the dynamic around the heteroclinic network is equivalent to the linearisation around each equilibrium.
By perturbing the coefficient $d$, the invariant subspaces that supported the heteroclinic connections are destroyed, and the heteroclinic network is broken.
When the heteroclinic network is stable, we show that, depending on the sign of $d$, the system exhibits one or three attracting periodic orbits close to it.
In both cases, the set of periodic orbits shadows the heteroclinic network.
When the heteroclinic network is not stable, the perturbed system presents one or two repelling periodic orbits, which cover only half of the destroyed heteroclinic network.
We present some numerical simulations to illustrate the results.
The symmetries of the system play an important role in the analysis of the dynamics on the invariant sphere and its perturbations.
In the fundamental domain, the heteroclinic network corresponds to two homoclinic orbits around the same equilibrium.
So we also have two Poincaré maps transversal to each homoclinic orbit, one orientation-preserving and the other orientation-reversing.
This means that a neighbourhood of the heteroclinic network correspond to an annulus glued with a Möbius band in the fundamental domain.
One homoclinic orbit is inside the annulus, and the other is inside the Möbius band.

In the literature, there has been significant research on the bifurcations of homoclinic orbits and heteroclinic networks in different settings \cite{HS10}. 
It is well known that co-dimension one bifurcations of one homoclinic orbit in the plane lead to the emergence of a unique periodic orbit in the annulus neighbourhood of the homoclinic orbit.
The arc length of the emerging periodic orbits converges to the arc length of the homoclinic orbit.
Moreover, under standard conditions, the co-dimension one bifurcation of one homoclinic orbit in $\mathbb{R}^n$ also leads to a unique periodic orbit if the leading eigenvalues are real, \cite{S68}.
However, the arc length of the emerging periodic orbits may converge to something different than the arc length of the homoclinic orbit.
Since the periodic orbits are given by fixed and periodic points of the Poincaré map transversal to the homoclinic orbit, the arc length of the emerging periodic orbits is closely related to the periodicity of the corresponding point in the Poincaré map.
For example, if the homoclinic orbit is placed in a Möbius band, then the Poincaré map is an orientation-reversing map, and the unique emerging periodic orbit is given by a periodic point with periodicity two.
So the arc length of the emerging periodic orbits converges to twice the arc length of the homoclinic orbit. 
An extreme case occurs when the homoclinic orbit is placed in a Klein bottle, and the arc length of the emerging periodic orbits converges to infinity.
Note that the periodicity of the Poincaré map is closely related to the number of times that the emerging periodic orbits pass near the equilibrium before closing the orbit. 

In the case study analysed here, we have two homoclinic orbits around the same equilibrium, one in the annulus and the other in the Möbius band.
The Poincaré maps transversal to each homoclinic orbit must be carefully composed to describe the dynamics in the fundamental domain.
This means that we do not have a simple concept of periodicity of the Poincaré map.
We will instead count the number of times that the emerging periodic orbits pass near the equilibria before closing the orbit.  
When the homoclinic orbits are stable, the perturbation gives rise to two attracting periodic orbits, one around the annulus and one around the Möbius band, or one periodic orbit which goes around the annulus and the Möbius band.
Thus the arc length of the attracting periodic orbits converges to the arc length of the homoclinic network.
When the homoclinic orbits are not stable, the perturbed system exhibits one repelling periodic orbit that may go around the annulus or the Möbius band. 
So, the arc length of the emerging periodic orbits converges to half the arc length of the homoclinic network.
This analysis provides the geometric intuition behind the results on the number and arc length of periodic orbits on the sphere arising from the destruction of the heteroclinic network.

The text is organised as follows. In Section~\ref{sec:inv_sphere}, we present the invariant sphere theorem and apply it to ring-coupled systems. In Section~\ref{sec:hetnetwork}, we study the existence of heteroclinic networks on the invariant sphere. In Section~\ref{sec:perturb}, we study perturbations of the $3$-ring-coupled system.

\section{Invariant sphere theorem}\label{sec:inv_sphere}

In this section, we will recall the invariant sphere theorem and the definition of ring-coupled systems. Then we will provide conditions on the cubic ring-coupled systems to guarantee the existence of an invariant sphere.

Let $P^{k}$ be the set of homogeneous polynomials $Q:\mathbb{R}^n\rightarrow \mathbb{R}^n$  with degree $k$.

\begin{defi}
An homogeneous polynomial $Q\in P^{2p+1}$ is said to be contracting if 
$$\langle Q(u),u\rangle<0,$$
for every $u\in S^{n-1}=\{x\in \mathbb{R}^n: \|x\|=1\}$.
\end{defi}

The invariant sphere theorem states that if $Q\in P^{2p+1}$ is contracting, then there exists an invariant sphere for the system $\dot{x}=\lambda x + Q(x)$, for every $\lambda>0$.

\begin{theo}(\cite[Theorem~5.1.1]{F07})
Let $p\geq 1$ and suppose that $Q\in P^{2p+1}$ is contracting. Then, for every $\lambda>0$ there exists a unique $n$-dimensional sphere $S(\lambda)\subset \mathbb{R}^n\setminus \{0\}$ which is invariant by the flow of 
$$\dot{x}=\lambda x + Q(x).$$
Further
\begin{enumerate}
	\item $S(\lambda)$ is globally attracting in the sense that every trajectory $x(t)$ with non-zero initial condition is asymptotic to $S(\lambda)$ as $t\rightarrow \infty$.
\item $S(\lambda)$ is embedded as a topological submanifold of $\mathbb{R}^n$ and the bounded component of $\mathbb{R}^n\setminus S(\lambda)$ contains the origin.
\item The flow restricted to $S(\lambda)$ is topologically equivalent to the flow of the phase vector field $\mathcal{P}_Q$, where $\mathcal{P}_Q(u)= Q(u)-(Q(u),u)u$ and $u\in S^{n-1}$.
\end{enumerate}
\end{theo}

Consider $N$ to be a ring network with $n$ cells and a single cycle connecting the cells, see Figure~\ref{fig:ringnetwork}.

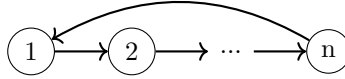
\begin{figure}[ht]
\begin{center}
\begin{tikzpicture}
\matrix [column sep=7mm, row sep=5mm] {
\node (v1) [draw, shape=circle] {1} ; &  \node (v2) [draw, shape=circle] {2}; &\node (v3) [ shape=circle] {...} ; &  \node (v4) [draw, shape=circle] {n}; \\
};
\draw[->, thick] (v1) to (v2);
\draw[->, thick] (v2) to (v3);
\draw[->, thick] (v3) to (v4);
\draw[->, thick] (v4) to [bend right] (v1);
\end{tikzpicture}
\end{center}
\caption{Ring network with $n$ cells}
\label{fig:ringnetwork}
\end{figure}

Following \cite{GST05}, the ring-coupled systems associated with a ring network have the form
$$
\dot{x}=F_N(x)\Leftrightarrow \begin{cases}
\dot{x}_1=f(x_1,x_n)\\
\dot{x}_2=f(x_2,x_1)\\
\vdots\\
\dot{x}_n=f(x_n,x_{n-1})
\end{cases}$$
where $x=(x_1,x_2,\dots,x_{n-1},x_n)\in\mathbb{R}^n$ and $f:\mathbb{R}^2\rightarrow \mathbb{R}$ is a function called the cell's dynamics. 
The vector-valued function $F_N:\mathbb{R}^n\rightarrow \mathbb{R}^n$ that respects the ring structure is defined by the function $f$. 
Denote by $\mathfrak{C}_N$ be the set of vector-valued functions that respects the structure of a ring $N$.
Let $P_N^{k}=\mathfrak{C}_N\cap P^{k}$ be the set of homogeneous polynomials with degree $k$ that also respects the structure of a ring $N$.
For each $Q_N:\mathbb{R}^n\rightarrow \mathbb{R}^n\in P_N^{k}$ there exists a homogeneous polynomial $q:\mathbb{R}^2\rightarrow \mathbb{R}$ with degree $k$.

We will focus on the case $k=3$ to consider cubic polynomials.
The homogeneous polynomials $q:\mathbb{R}^2\rightarrow \mathbb{R}$ with degree $3$ are given by the following
$$q(y,z)=a y^3+b y^2z+c yz^2+d z^3,$$
for some constants $a,b,c,d\in\mathbb{R}$. Let $Q_N$ be the corresponding ring-coupled cubic polynomial given by 
$$Q_N(x) = (q(x_1,x_n), q(x_2,x_1), \dots, q(x_n,x_{n-1})).$$
In order to apply the invariant sphere theorem, the polynomial $Q_N(x)$ must be contracting, this means that 
$$\langle Q_N(x),x\rangle<0,$$ 
for every $x\in\mathbb{R}^n\setminus \{0\}$.
The following result provides sufficient conditions for $Q_N$ to be contracting.

\begin{prop}\label{prop:contrac}
Let $q(y,z)=a y^3+b y^2z+c yz^2+d z^3$ be a cubic homogeneous polynomial.

If there exists $k$ such that $a<k<0$ and 
$$(\Delta(k)>0 \wedge (P(k)> 0 \vee D(k)> 0)) \vee (\Delta(k)=D(k)=R(k)=0 \wedge P(k)>0),$$
where 
$$\begin{aligned}
\Delta(k) ={}&256(a-k)^{3}k^{3}-192(a-k)^{2}bdk^{2}-128(a-k)^{2}c^{2}k^{2}+144(a-k)^{2}cd^{2}k-27(a-k)^{2}d^{4}\\
&+144(a-k)b^{2}ck^{2}-6(a-k)b^{2}d^{2}k-80(a-k)bc^{2}dk+18(a-k)bcd^{3}+16(a-k)c^{4}k\\
&-4(a-k)c^{3}d^{2}-27b^{4}k^{2}+18b^{3}cdk-4b^{3}d^{3}-4b^{2}c^{3}k+b^{2}c^{2}d^{2}
\end{aligned}$$
$$P(k)=8(a-k)c-3b^{2}$$
$$D(k)=64(a-k)^{3}k-16(a-k)^{2}c^{2}+16(a-k)b^{2}c-16(a-k)^{2}bd-3b^{4}$$
$$R(k)=b^{3}+8d(a-k)^{2}-4(a-k)bc$$
then $Q_N\in P_N^{3}$ is contracting.
\end{prop}

\begin{proof}
Writing the indices module $n$ ($x_0\equiv x_n$ and $x_{n+1}\equiv x_1$), we have that
$\sum_{i=1}^n x_{i-1}^4= \sum_{i=1}^n x_i^4$ and 
\begin{align*}
\langle Q_N(x),x\rangle=&\sum_{i=1}^n a x_i^4+b x_i^3x_{i-1}+c x_i^2x_{i-1}^2+d x_ix_{i-1}^3,\\
=&\sum_{i=1}^n (a-k) x_i^4+b x_i^3x_{i-1}+c x_i^2x_{i-1}^2+d x_ix_{i-1}^3+k x_{i-1}^4.
\end{align*}
where $k\in\mathbb{R}$.

Let $h_k(x,y)=(a-k) x^4+b x^3y+c x^2y^2+d x y^3+k y^4$ be the function that is summed before. 
If $h_k(x,y)$ is negative defined for some $k$, $h_k(x,y)< 0$ for every $(x,y)\in\mathbb{R}^2\setminus \{(0,0)\}$, then $Q_N$ is contracting.
It is possible to give sufficient and necessary condition for $h_k(x,y)$ to be negative defined
Since $h_k(\lambda (x,y))=\lambda^4 h_k(x,y)$, we have that 
$$h_k(x,y)< 0 \forall_{(x,y)\in\mathbb{R}^2\setminus \{(0,0)\}}$$
$$\Leftrightarrow h_k(\dfrac{x}{y},1)<0 \wedge h_k(1,0)<0 \forall_{x\in\mathbb{R}}\forall_{y\in \mathbb{R}\setminus\{0\}}$$
$$\Leftrightarrow h_k(t,1)<0 \wedge a<k \forall_{t\in\mathbb{R}}.$$
Let $g(t)=h_k(t,1)=(a-k) t^4+b t^3+c t^2+d t +k$. 
Then $g(t)$ is always negative, if and only if
$$g(0)<0 \textrm{ and $g$ has non real roots.}$$
Following, \cite{Y99,L88}, we obtain the following algebraic conditions
$$k<0 \wedge (\Delta(k)>0 \wedge (P(k)> 0 \vee D(k)> 0)) \vee (\Delta(k)=D(k)=R(k)=0 \wedge P(k)>0).$$
\end{proof}

Next, we exemplify the previous results and highlight the importance of the choice of $k$ in the previous proposition.

\begin{exa}
Let $a=-3$, $b=-2$, $c=-1$ and $d$ small.
For $k=-1$, we have that $P(-1)=4$ and
$$\Delta(-1)=2272 + 1712 d + 524 d^2 - 40 d^3 - 108 d^4.$$
So for small values of $d$, we have that $\Delta(-1)>0$ and $P(-1)>0$. 
Hence, by Proposition~\ref{prop:contrac}, $Q_N$ is contracting.

Note that $\Delta(0)=-8 d^2 - 76 d^3 - 243 d^4$ is negative for small values of $d$. So, $k$ must be carefully selected to guarantee that $\Delta(k)\geq 0$.
\end{exa}

We finish the first part of this work by putting together the previous results and the invariant sphere theorem to obtain the following corollary.

\begin{coro}
Let $N$ be a ring network with $n$ cells, and let $q(y,z)=a y^3+b y^2z+c yz^2+d z^3$ be a cubic homogeneous polynomial.

If there exists $k$ such that $a<k<0$ and 
$$(\Delta(k)>0 \wedge (P(k)> 0 \vee D(k)> 0)) \vee (\Delta(k)=D(k)=R(k)=0 \wedge P(k)>0)$$
then for every $\lambda>0$ there exists a unique $n$-dimensional sphere $S(\lambda)\subset \mathbb{R}^n\setminus \{0\}$ which is invariant by the flow of
$$\dot{x}=\lambda x + Q_N(x).$$
Moreover, the $S(\lambda)$ is globally attractive and the flow restricted to $S(\lambda)$ is topologically equivalent to the flow on the unitary sphere of the system
$$\dot{u}=\mathcal{P}_{Q_N}(u),$$
where $\mathcal{P}_{Q_N}(u)= Q_N(u)-\langle Q_N(u),u\rangle u$ and $u\in S^{n-1}$.
\end{coro}

Figure \ref{fig:ring 3 celulas sink} shows the dynamic on the simplex when $a=-3$, $b=-2$, $c=-1$, $d=0$ and $\lambda=0.1$.

\begin{figure}[ht]
	\centering
		\includegraphics[width=0.5\textwidth]{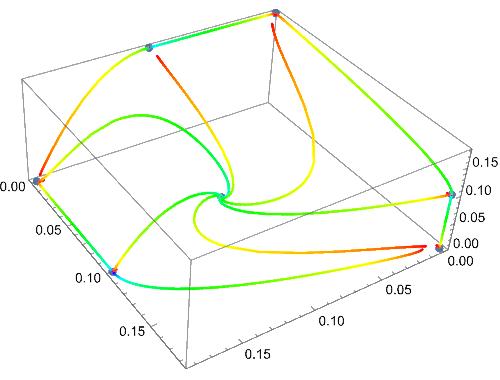}
	\caption{Dynamic on the simplex when $a=-3$, $b=-2$, $c=-1$, $d=0$ and $\lambda=0.1$}
	\label{fig:ring 3 celulas sink}
\end{figure}

Note that the condition $(\Delta(k)>0 \wedge (P(k)> 0 \vee D(k)> 0))$ is open in the coefficients $a,b,c,d$, this means that small perturbations of the coefficients $a,b,c,d$ do not break the previous condition and retains the existence of a invariant sphere.

\section{Heteroclinic networks}\label{sec:hetnetwork}

A heteroclinic cycle is a finite collection of equilibria $\{p_1,\dots,p_m\}$ and heteroclinic connections $\{\gamma_1,\dots,\gamma_{m}\}$, where $\gamma_i$ is a trajectory connecting $p_i$ to $p_{i+1}$, for $i=1,\dots,m-1$, and $\gamma_m$ connects $p_m$ to $p_1$. 
A heteroclinic network is a union of heteroclinic cycles that share some equilibria.

From now on, we assume that $Q_N$ is contracting and $\lambda>0$. 
In this case, we know that $a=\langle Q_N(1,0,\dots,0),(1,0,\dots,0)\rangle$ is negative. So we can make a time-scaling $t=-a\tilde{t}$ and assume without loss of generality that $a=-1$.
Along this section, we will assume that $d=0$ to guarantee the existence of invariant subspaces that support the heteroclinic connections. So, in this section, we are considering cubic homogeneous polynomial of the form $$q(y,z)=-y^3+b y^2 z + c y z^2,$$
where $b,c\in\mathbb{R}$.
By the invariant sphere theorem, the flow of the system (\ref{eq:edoring}) converges to an invariant sphere, and the dynamic on this sphere is equivalent to the following dynamical system on the unitary sphere 
\begin{equation}\label{eq:edosphere} 
\dot{u}=\mathcal{P}_N(u),
\end{equation}
where $\mathcal{P}_N(u)= Q_N(u)-\langle Q_N(u),u\rangle u$ and $u\in S^{n-1}$.

Since $d=0$, we have that $$A_i=\{u\in S^{n-1}: u_i=0\}$$ is invariant by the vector field $\mathcal{P}_N$ for every $i=1, \dots, n$.
The intersection $\bigcap_{i\neq k} A_i = \{ \pm e_k\}$, for any $k=1, \dots, n$, is also invariant by $\mathcal{P}_N$. Therefore the $2n$ poles, $\pm e_1, \pm e_2,\dots, \pm e_n$, are equilibria of the system (\ref{eq:edosphere}). 
We will use the following invariant circles in the sphere to connect the poles by heteroclinic connections:
$$B_i:=\bigcap_{j\neq i,i+1} A_j=\{u\in S^{n-1}: u_j=0, j\neq i,i+1\},$$
where $x_{n+1}\equiv x_1$. 
The following result states when there is a heteroclinic network connecting the poles along the previous circles.

\begin{prop}\label{prop:hetcycle}
Let $N$ be the $n$-ring, $\lambda>0$, and $q(y,z)=-y^3+b y^2 z + c y z^2$ such that $Q_N$ is contracting.
\begin{itemize}

\item[(i)] If $b^2+4(c+1)<0$ or $(b,c)=(0,-1)$, then the system (\ref{eq:edoring}) has a heteroclinic network connecting the poles $\pm e_1, \pm e_2,\dots, \pm e_n$ with heteroclinic connections from $e_{i+1}$ to $\pm e_i$ and from $-e_{i+1}$ to $\pm e_i$ for every $i=1, \dots, n$. 

\item[(ii)] If $b^2+4(c+1)\geq 0$ and $c \leq -1$ and $(b,c)\neq (0,-1)$, then the system (\ref{eq:edoring}) has one or two heteroclinic cycles depending on the values of $b$ and $n$. Specifically, 
\begin{itemize}
\item[(ii.1)] If $b<0$, there are two heteroclinic cycles $\{e_n,\dots, e_2,e_1\}$ and $\{ -e_n,\dots, -e_2, -e_1\}$.

\item[(ii.2)] If $b>0$ and $n$ is even, there are two heteroclinic cycles $\{e_n,-e_{n-1},\dots, e_2,-e_1\}$ and $\{-e_n,e_{n-1},\dots, -e_2,e_1\}$.

\item[(ii.3)] If $b>0$ and $n$ is odd, there is one heteroclinic cycle \\$\{e_n,-e_{n-1},\dots, -e_2,e_1, -e_n, e_{n-1},\dots, e_2,-e_1\}$.  
\end{itemize}
\end{itemize}
\end{prop}

\begin{proof}

The previous result follows from a detailed analysis of the dynamics on the invariant circles $B_i$. The dynamic on $B_i$ is illustrated in Figure~\ref{fig:dynonbi} for $b<0$.
When $b>0$, the dynamic on $B_i$ is similar but Figures~\ref{fig:dynonbi2},\ref{fig:dynonbi3} and \ref{fig:dynonbi4} are reflected with respect to the $e_i$ axis. This means that equilibrium point between the poles $e_i$ and $-e_{i+1}$ are turned into equilibrium point between the poles $e_i$ and $e_{i+1}$, and the equilibrium point between the poles $-e_i$ and $e_{i+1}$ are turned into equilibrium point between the poles $-e_i$ and $-e_{i+1}$. The previous proposition easily follows from the dynamics described in Figure~\ref{fig:dynonbi}.

\begin{figure}[ht]
\centering
\begin{subfigure}{0.15\textwidth}
\begin{center}
\begin{tikzpicture}
\node (ei) at (0,1)   [draw, circle,color=black,fill=black ,inner sep=1pt, label = above:{\small$e_i$}] {}; 
\node (ei1) at (1,0)   [draw, circle,color=black,fill=black ,inner sep=1pt, label = left:{\small$e_{i+1}$}]{};
\node (mei) at (0,-1)   [draw, circle,color=black,fill=black ,inner sep=1pt, label = below:{\small-$e_i$}]{}; 
\node (mei1) at (-1,0)   [draw, circle,color=black,fill=black ,inner sep=1pt, label = right:{\!\!\small-$e_{i+1}$}]{};

\node (ai) at (0.71,0.71)   [draw, circle,color=black,fill=black ,inner sep=1pt, label = left:{ }]{};
\node (ai1) at (0.71,-0.71)   [draw, circle,color=black,fill=black ,inner sep=1pt, label = left:{ }]{}; 
\node (mai) at (-0.71,-0.71)   [draw, circle,color=black,fill=black ,inner sep=1pt, label = left:{ }]{};
\node (mai1) at (-0.71,0.71)   [draw, circle,color=black,fill=black ,inner sep=1pt, label = left:{ }]{};

\draw[color=black,line width=0.2mm] (0,0) circle (1);

\draw[->] (1,0) arc (0:25:1);
\draw[->] (1,0) arc (0:-25:1);

\draw[->] (-1,0) arc (180:205:1);
\draw[->] (-1,0) arc (180:155:1);

\draw[->] (0,1) arc (90:65:1);
\draw[->] (0,1) arc (90:115:1);

\draw[->] (0,-1) arc (270:295:1);
\draw[->] (0,-1) arc (270:245:1);
\end{tikzpicture}
\end{center}
\caption{ }\label{fig:dynonbi1}
\end{subfigure}
\begin{subfigure}{0.15\textwidth}
\begin{center}
\hspace{5mm}
\begin{tikzpicture}
\node (ei) at (0,1)   [draw, circle,color=black,fill=black ,inner sep=1pt, label = above:{\small$e_i$}] {}; 
\node (ei1) at (1,0)   [draw, circle,color=black,fill=black ,inner sep=1pt, label = left:{\small$e_{i+1}$}]{};
\node (mei) at (0,-1)   [draw, circle,color=black,fill=black ,inner sep=1pt, label = below:{\small-$e_i$}]{}; 
\node (mei1) at (-1,0)   [draw, circle,color=black,fill=black ,inner sep=1pt, label = right:{\!\!\small-$e_{i+1}$}]{};

\node (ai1) at (0.71,-0.71)   [draw, circle,color=black,fill=black ,inner sep=1pt, label = left:{ }]{}; 
\node (mai1) at (-0.71,0.71)   [draw, circle,color=black,fill=black ,inner sep=1pt, label = left:{ }]{};

\draw[color=black,line width=0.2mm] (0,0) circle (1);

\draw[->] (ei1) arc (0:45:1);
\draw[->] (ei1) arc (0:-25:1);

\draw[->] (mei1) arc (180:225:1);
\draw[->] (mei1) arc (180:155:1);

\draw[->] (ei) arc (90:115:1);

\draw[->] (mei) arc (270:295:1);
\end{tikzpicture}
\end{center}
\caption{ }\label{fig:dynonbi2}
\end{subfigure}
\begin{subfigure}{0.15\textwidth}
\begin{center}
\begin{tikzpicture}
\node (ei) at (0,1)   [draw, circle,color=black,fill=black ,inner sep=1pt, label = above:{\small$e_i$}] {}; 
\node (ei1) at (1,0)   [draw, circle,color=black,fill=black ,inner sep=1pt, label = left:{\small$e_{i+1}$}]{};
\node (mei) at (0,-1)   [draw, circle,color=black,fill=black ,inner sep=1pt, label = below:{\small-$e_i$}]{}; 
\node (mei1) at (-1,0)   [draw, circle,color=black,fill=black ,inner sep=1pt, label = right:{\!\!\small-$e_{i+1}$}]{};

\node (bi) at (0.87,-0.5)   [draw, circle,color=black,fill=black ,inner sep=1pt, label = left:{ }]{};
\node (bi1) at (0.5,-0.87)   [draw, circle,color=black,fill=black ,inner sep=1pt, label = left:{ }]{}; 
\node (mai) at (-0.87,0.5)   [draw, circle,color=black,fill=black ,inner sep=1pt, label = left:{ }]{};
\node (mai1) at (-0.5,0.87)   [draw, circle,color=black,fill=black ,inner sep=1pt, label = left:{ }]{};

\draw[color=black,line width=0.2mm] (0,0) circle (1);

\draw[->] (ei1) arc (0:45:1);
\draw[->] (ei1) arc (0:-20:1);

\draw[->] (mei1) arc (180:165:1);
\draw[->] (mei1) arc (180:225:1);

\draw[->] (mai1) arc (120:100:1);
\draw[->] (mai1) arc (120:140:1);

\draw[->] (bi1) arc (300:320:1);
\draw[->] (bi1) arc (300:280:1);
\end{tikzpicture}
\end{center}
\caption{ }\label{fig:dynonbi3}
\end{subfigure}
\begin{subfigure}{0.15\textwidth}
\begin{center}
\begin{tikzpicture}
\node (ei) at (0,1)   [draw, circle,color=black,fill=black ,inner sep=1pt, label = above:{\small$e_i$}] {}; 
\node (ei1) at (1,0)   [draw, circle,color=black,fill=black ,inner sep=1pt, label = left:{\small$e_{i+1}$}]{};
\node (mei) at (0,-1)   [draw, circle,color=black,fill=black ,inner sep=1pt, label = below:{\small-$e_i$}]{}; 
\node (mei1) at (-1,0)   [draw, circle,color=black,fill=black ,inner sep=1pt, label = right:{\!\!\small-$e_{i+1}$}]{};

\node (ai1) at (0.71,-0.71)   [draw, circle,color=black,fill=black ,inner sep=1pt, label = left:{ }]{}; 
\node (mai1) at (-0.71,0.71)   [draw, circle,color=black,fill=black ,inner sep=1pt, label = left:{ }]{};

\draw[color=black,line width=0.2mm] (0,0) circle (1);

\draw[->] (ei1) arc (0:45:1);
\draw[->] (ei1) arc (0:-25:1);

\draw[->] (mei1) arc (180:155:1);

\draw[->] (mai1) arc (135:105:1);

\draw[->] (ai1) arc (315:290:1);
\draw[->] (mei1) arc (180:220:1);
\end{tikzpicture}\end{center}
\caption{ }\label{fig:dynonbi4}
\end{subfigure}
\begin{subfigure}{0.15\textwidth}
\begin{center}
\begin{tikzpicture}
\node (ei) at (0,1)   [draw, circle,color=black,fill=black ,inner sep=1pt, label = above:{\small$e_i$}] {}; 
\node (ei1) at (1,0)   [draw, circle,color=black,fill=black ,inner sep=1pt, label = left:{\small$e_{i+1}$}]{};
\node (mei) at (0,-1)   [draw, circle,color=black,fill=black ,inner sep=1pt, label = below:{\small-$e_i$}]{}; 
\node (mei1) at (-1,0)   [draw, circle,color=black,fill=black ,inner sep=1pt, label = right:{\!\!\small-$e_{i+1}$}]{};

\draw[color=black,line width=0.2mm] (0,0) circle (1);

\draw[->] (ei1) arc (0:45:1);
\draw[->] (ei1) arc (0:-45:1);

\draw[->] (mei1) arc (180:225:1);
\draw[->] (mei1) arc (180:135:1);

\end{tikzpicture} \end{center}
\caption{ }\label{fig:dynonbi5}
\end{subfigure}
\caption{Dynamics on the circle $B_i$ when $b<0$ and (\subref{fig:dynonbi1}) $c>-1$; (\subref{fig:dynonbi2}) $c=-1$; (\subref{fig:dynonbi3}) $b^2+4(c+1)>0$ and $c<-1$; (\subref{fig:dynonbi4}) $b^2+4(c+1)=0$; (\subref{fig:dynonbi5}) $b^2 + 4(c + 1) < 0$.}
\label{fig:dynonbi}
\end{figure}
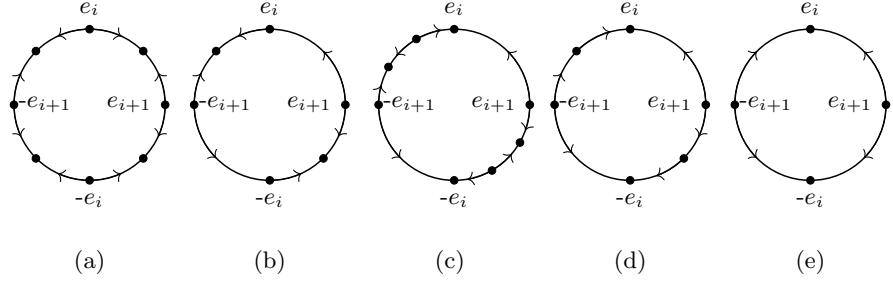

In the rest of this proof, we will show that the dynamics presented in Figure~\ref{fig:dynonbi} is accurate. This analysis is  done for the circle $B_1$ and then translated to the other circles $B_i$ by the $\mathbb{Z}_n$ symmetry. 
We already know that the poles $\pm e_1$ and $\pm e_2$ are equilibria of the system (\ref{eq:edoring}).
We look for other equilibria of the system (\ref{eq:edoring}) in the invariant circle $B_1$. For $u=(u_1,u_2,0,\dots,0)\in B_1\subset S^{n-1}$ with $u_1u_2\neq 0$, we have that $|u_1|,|u_2|<1$, $u_1^2+u_2^2=1$ and
$$\mathcal{P}_N(u)=0
\Leftrightarrow 
q(u_2,u_1) = - u_1^2u_2 
\Leftrightarrow 
- u_2^2+b u_2u_1+(c+1) u_1^2= 0.
$$

So $u^*=(u_1^*,u_2^*,0,\dots,0)\in B_1\subset S^{n-1}$ with $u_1u_2\neq 0$ is a equilibrium point if and only if  the quadratic form $-u_2^2+b u_2u_1+(c+1) u_1^2$ vanishes along the line $u_2=m u_1$, where $m=\dfrac{u_2^*}{u_1^*}\neq 0$,
 $$- m^2 u_1^2+b m u_1^2+(c+1) u_1^2= 0\Leftrightarrow  m^2 - b m-(c+1) = 0\Leftrightarrow m=\dfrac{b\pm\sqrt{b^2+4(c+1)}}{2}.$$

Thus there are other equilibrium apart from the poles if and only if $b^2+4(c+1)\geq 0$.
If $c>-1$, then $b^2+4(c+1)>|b|\geq 0$ and two slopes $m$ have different signs, so there are four other equilibria in the circle $B_1$ as depicted in Figure~\ref{fig:dynonbi1}.
If $c=-1$ and $b\neq 0$, then $b^2+4(c+1)=b^2\geq 0$ and there is only one non vanishing slope $m=2b$. So there are two other equilibria in the circle $B_1$ as depicted in Figure~\ref{fig:dynonbi2} when $b<0$. When $b>0$, the equilibria appear by reflection with respect to the $e_1$ axis.
If $b^2+4(c+1)>0$ and $c<-1$, then the two slopes $m$ have the same sign. So there are four other equilibria in the circle $B_1$ as depicted in Figure~\ref{fig:dynonbi3} when $b<0$. When $b>0$, the equilibria appear by reflection with respect to the $e_1$ axis.
If $b^2+4(c+1)=0$, then there is one non-vanishing slope $m=\dfrac{b}{2}$ with the same sign that $b$. So there are two equilibria as depicted in Figure~\ref{fig:dynonbi4} when $b<0$. When $b>0$, the equilibria appear by reflection with respect to the $e_1$ axis.
If $b^2+4(c+1)<0$ or $(b,c)=(0,-1)$, then there is no other equilibrium in the circle $B_1$ apart from the poles as depicted in Figure~\ref{fig:dynonbi5}.

Now, we will analyse the stability of the poles in the system (\ref{eq:edosphere}).
Using the $\mathbb{Z}_2\times\mathbb{Z}_n$ symmetry, it is enough to study the stability of the pole $e_1$.
We have the following lemma that relates the linearization of $Q_N$ and $\mathcal{P}_N$.

\begin{lem}\label{lem:spherelinear}(\cite[Lemma 4.6.2]{F07})
The linearization of $\mathcal{P}_N$ at an equilibrium point $u\in S^{n-1}$ is given by:
$$\mathcal{T}_u\mathcal{P}_N=D'Q_N(u)-\langle Q_N(u),u\rangle Id_u,$$
where $Id_u$ is the identity on the tangent space $T_u S^{n-1}=\{v\in\mathbb{R}^n: \langle u,v\rangle=0\}$ and $D'Q_N(u)$ is $DQ_N(u)$ restricted to $T_u S^{n-1}$.

Let $\mu_1,\dots,\mu_k$ be the eigenvalues of $D'Q_N(u)$, the eigenvalues of $\mathcal{T}_u\mathcal{P}_N$ are given by 
$$\mu_1-\langle Q_N(u),u\rangle,\dots, \mu_k-\langle Q_N(u),u\rangle.$$
\end{lem}

The derivatives of $q(y,z)$ with respect to $y$ and $z$ are, respectively:
$$q_0(y,z)=3a y^2+2b yz+c z^2=-3 y^2+2b yz+c z^2,$$
$$q_1(y,z)=b y^2+2c yz+3d z^2=b y^2+2c yz.$$

Consider the equilibrium $e_1=(1,0,\dots,0)$. Note that $\langle Q_N(e_1),e_1\rangle=a=-1$ and
\begin{equation}\label{eq:DQNu1}
DQ_N(e_1)=
\begin{bmatrix}
q_0(1,0)&0&\dots&q_1(1,0)\\
q_1(0,1)&q_0(0,1)&\dots&0\\
\vdots & \vdots&\ddots&\vdots\\
0&0&\dots&0\\
\end{bmatrix}=\begin{bmatrix}
-3&0&\dots&b\\
0&c&\dots&0\\
\vdots & \vdots&\ddots&\vdots\\
0&0&\dots&0\\
\end{bmatrix}.
\end{equation}

The matrix $D'Q_N(u)$ is the submatrix of $DQ_N(u)$ by removing the first line and the first column.
The eigenvalues of $D'Q_N(u)$ are $c,0,\dots,0$.
Following Lemma~\ref{lem:spherelinear}, the eigenvalues of $\mathcal{T}_{e_1}\mathcal{P}_N$ are
$$c+1,1,\dots,1.$$
Note that the eigenvector associated with eigenvalue $c+1$ belongs to the space $\mathcal{T}_{e_1} S^{n-1}\cap B_1$.
Thus, the coefficient $c$ determines the stability of the poles $\pm e_1$ in the invariant space $B_1$.
While the equilibrium $\pm e_2$ is always a source in $B_1$ since the other eigenvalues of $\mathcal{T}_{e_1}\mathcal{P}_N$ are positive.
Since the circle is one-dimensional, the direction of the dynamics on the invariant circle $B_1$ is the one displayed in Figure~\ref{fig:dynonbi}.
In fact, the direction between the two middle equilibria in Figure~\ref{fig:dynonbi3} was not discussed before. 
However, this is not relevant for this proof.
\end{proof}

Note that there is no heteroclinic connection between two non-consecutive poles, since the stable manifold of each pole is at most one-dimensional and it is included in the invariant spaces $B_i$.

\begin{exa}
We illustrate the different behaviours in Proposition~\ref{prop:hetcycle} for the case of three coupled cells, $n=3$.

Taking $(b,c)=(0.25,-3)$, we have that $b^2+4(c+1)<0$. So the heteroclinic network connects all the poles as depicted in Figure~\ref{fig:hetcyclesexa1}. 
For $(b,c)=(-3,-3)$, we have that $b^2+4(c+1)>0$, $c<-1$ and $b<0$. So there are two heteroclinic cycles $\{e_3,e_2,e_1\}$ and $\{-e_3,-e_2,-e_1\}$ as shown in Figure~\ref{fig:hetcyclesexa2}.
When $(b,c)=(3,-3)$, we have that $b^2+4(c+1)>0$,
$c<-1$, $b>0$ and $n$ is odd. So there is one heteroclinic cycle $\{e_3,-e_2,e_1, -e_3, e_2,-e_1\}$ connecting the poles as depicted in Figure~\ref{fig:hetcyclesexa3}
Taking $k=-0.15$ in Proposition~\ref{prop:contrac}, we know that $Q_N$ is contracting in those cases.

\end{exa}
\begin{figure}[ht]
\centering
\begin{subfigure}{0.3\textwidth}
\includegraphics[width=\textwidth]{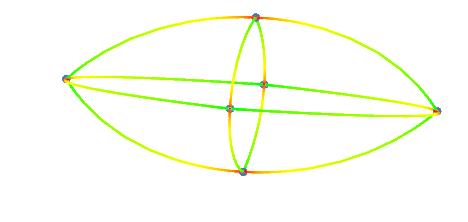}
\caption{$(0.25,3)$ }\label{fig:hetcyclesexa1}
\end{subfigure}
\begin{subfigure}{0.3\textwidth}
\includegraphics[width=\textwidth]{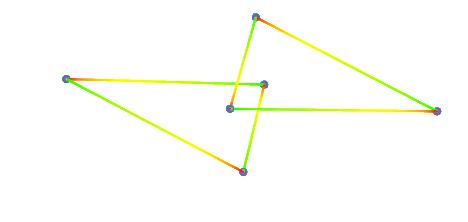}
\caption{$(-3,-3)$ }\label{fig:hetcyclesexa2}
\end{subfigure}
\begin{subfigure}{0.3\textwidth}
\includegraphics[width=\textwidth]{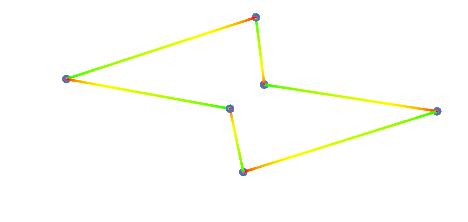}
\caption{$(3,-3)$ }\label{fig:hetcyclesexa3}
\end{subfigure}

\caption{Numerical simulation of $\dot{x}=\lambda x + Q_N(x)$ for a 3-coupled ring with $\lambda=1$. 
The parameters $(b,c)$ are indicated in the respective figure. These simulations illustrate the different behaviours in Proposition~\ref{prop:hetcycle}.}
\label{fig:hetcyclesexa}
\end{figure}

The space of coefficients $(b,c)\in\mathbb{R}^2$ is divided into three regions by the curve $b^2+4(c+1)=0$ and the line $c=-1$. Figure~\ref{fig:divbcBi} places the dynamics on the invariant circles $B_i$ according with Figure~\ref{fig:dynonbi} in each region of the space of coefficients $(b,c)$. 

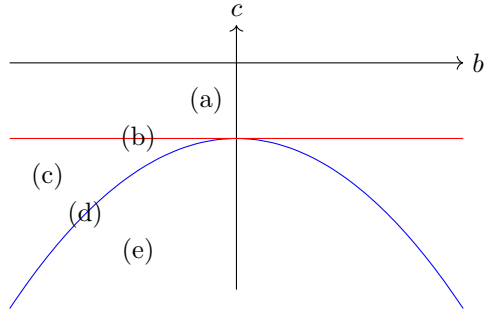
\begin{figure}[ht]
\centering
\begin{tikzpicture}
  \draw[->] (-3, 0) -- (3, 0) node[right] {$b$};
  \draw[->] (0, -3) -- (0, 0.5) node[above] {$c$};
  \draw[scale=1, domain=-3:3, smooth, variable=\b, blue] plot ({\b}, {-1-\b*\b/4});
  \draw[scale=1, domain=-3:3, smooth, variable=\b, red]  plot ({\b}, {-1});

  \node (a) at (-0.4,-0.5)   [circle] {(a)}; 
	\node (b) at (-1.3,-1)   [circle] {(b)}; 
	\node (c) at (-2.5,-1.5)   [circle] {(c)}; 
	\node (d) at (-2,-2)   [circle] {(d)}; 
	\node (e) at (-1.3,-2.5)   [circle] {(e)}; 
\end{tikzpicture}
\caption{Division of the parameters $b,c$. The letters (a), (b), (c), (d), (e) indicate the regions where the respective dynamic described in Figure~\ref{fig:dynonbi} occurs.}
\label{fig:divbcBi}
\end{figure}

Recalling the division of parameters in Figure~\ref{fig:divbcBi} and the dynamics at the circle $B_i$ represented in Figure~\ref{fig:dynonbi}. Except on the point $(b,c)=(0,-1)$, a transcritical bifurcation occurs when the parameters cross the line $c=-1$ at the equilibria, $\pm e_i$, and a fold bifurcation occurs when they cross the line $b^2+4(c+1)=0$ at the equilibria $\pm(u_1,u_2,0,\dots,0)\in S^{n-1}$ such that $u_2=u_1 b/2$.
The linearization of $\mathcal{P}_N$ around the equilibria $\pm e_i$ is invertible in the complementary of $B_i$ in $S^{n-1}$
In the next remark, we check that linearisation of $\mathcal{P}_N$ around the equilibria $\pm(u_1,u_2,0,\dots,0)\in S^{n-1}$ such that $u_2=u_1 b/2$ is invertible in the complement of $B_i$ in $S^{n-1}$, if $b\neq \pm \sqrt{2 \left(\sqrt{5}-1\right)}$.
Thus, the steady states appear and disappear exclusively on the invariant space $B_i$.
And there exists a fold bifurcation in the invariant space $B_i$ when the parameters cross the line $b^2+4(c+1)=0$, except at the three points $0,\pm \sqrt{2 \left(\sqrt{5}-1\right)}$.

\begin{rem}
Let $b,c\in\mathbb{R}$ such that $b^2+4(c+1)=0$ and $b\neq 0$.  
Denote the equilibrium on (\ref{eq:edoring}) by  $u^*=(u_1^*,u_2^*,0,\dots,0)\in S^{n-1}$ such that $u_2^*=u_1^* b/2$.
We have that $\pm u_1^*= \pm\frac{2}{\sqrt{b^2+4}}$,  
$$DQ_N(\pm u^*)=\begin{bmatrix}
-\frac{12}{b^2+4}& 0 &0&\dots&0&\frac{4 b}{b^2+4}\\
-\frac{4 b}{b^2+4}& -\frac{4}{b^2+4} &0&\dots&0&0\\
0& 0 &-\frac{b^2}{4}&\dots&0&0\\
\vdots& \vdots &\vdots&\ddots&\vdots&\vdots\\
0& 0 &0&\dots&0&0\\
0& 0 &0&\dots&0&0
\end{bmatrix},$$
so, $DQ_N(u^*)u^*=-\frac{12}{b^2+4} u^* $ and $\langle Q_N(u^*),u^*\rangle=-\frac{4}{b^2+4}$.
The Jacobian of $Q_N$ has a block structure, where the first two columns and lines correspond to the directions of $u^*$ and $B_i$. The eigenvalues of $DQ_N(\pm u^*)$ are $-\frac{12}{b^2+4}, -\frac{4}{b^2+4}, -\frac{b^2}{4},0,\dots,0$.
Since the vector $u^*$ is an eigenvector associated with the eigenvalue $-\frac{12}{b^2+4}$, it follows from Lemma~\ref{lem:spherelinear} that the linearisation of $\mathcal{P}_N$ around $u^*$ has the following eigenvalues:
$$ 0, -\frac{b^2}{4}+\frac{4}{b^2+4},\frac{4}{b^2+4},\dots,\frac{4}{b^2+4}.$$

If $-\frac{b^2}{4}+\frac{4}{b^2+4}\neq 0\Leftrightarrow b\neq \pm \sqrt{2 \left(\sqrt{5}-1\right)}$, there is only one vanishing eigenvalue which has the associated eigenvectors in the tangent space $T_{\pm u^*} S^{n-1}\cap B_{1}$.
So along the line $b^2+4(c+1)=0$ and except in three points $0,\pm \sqrt{2 \left(\sqrt{5}-1\right)}$, the linearisation of $\mathcal{P}_N$ around the equilibria $\pm(u_1,u_2,0,\dots,0)\in S^{n-1}$ such that $u_2=u_1 b/2$ is invertible in the complement of $B_i$ in $S^{n-1}$.
\end{rem}

From the point of view of the equivariant theory, the previous heteroclinic cycles can be seen as homoclinic connections between the same equilibrium.
The poles $\pm e_1, \pm e_2, \dots, \pm e_n$ are in the same orbit by the action of the group $\mathbb{Z}_n\times \mathbb{Z}_2$. Moreover, the connections $e_i\rightarrow e_{i+1}$ are mapped to the connections $\pm e_{i+1}\rightarrow \pm e_{i+2}$ by the action of $\mathbb{Z}_n\times \mathbb{Z}_2$.
And the connections $e_i\rightarrow -e_{i+1}$ are mapped to the connections $\pm e_{i+1}\rightarrow \mp e_{i+2}$ by the action of $\mathbb{Z}_n\times \mathbb{Z}_2$.
So the heteroclinic network given by Proposition~\ref{prop:hetcycle}~(i) corresponds to two homoclinic connections around the same equilibrium.
And the heteroclinic cycles given by Proposition~\ref{prop:hetcycle}~(ii) correspond to just one homoclinic connection around the same equilibrium.
Note that the class of $\mathbb{Z}_n\times \mathbb{Z}_2$-equivariant vector fields is bigger than the class of $n$-ring-coupled systems.

\subsection{Stability of the heteroclinic network}

The heteroclinic network is asymptotically stable if initial conditions close enough to the heteroclinic network always stay close to the heteroclinic network and  converge to the network as time passes.

\begin{defi}{\cite[Definition 2.3.]{KM95}}
A heteroclinic network $\Sigma$ is said to be asymptotically stable if for any neighbourhood $U$ of $\Sigma$, there exists a smaller neighbourhood $V$ such that trajectories starting in $V$ remain in $U$ for all forward time and they are asymptotic to $\Sigma$.
\end{defi}

\begin{prop}
Let $N$ be the $n$ ring and $q(y,z)=- y^3+b y^2 z + c y z^2$ such that $Q_N$ is contracting. Suppose that $b^2+4(c+1)<0$. The heteroclinic network given in Proposition~\ref{prop:hetcycle} is asymptotically stable if and only if $n=3$ and $c<-2$.
\end{prop}

\begin{proof}
In order to be stable, the heteroclinic cycle must not have expanding directions out of the cycle. When $n>3$, the equilibria $e_j$ and $e_k$ are sources in the flow-invariant one-dimensional space $C_{j,k}=\bigcap_{i\neq j,k}{A_i}$ since the eigenvalue of $\mathcal{T}_{e_i} S^{n-1}\cap B_n$ is $-a>0$ when $k\neq j-1,j+1$. Recall that there is an equilibrium point between $e_j$ and $e_k$ in $C_{j,k}$.
So there exists an expanding direction out of the cycle at each equilibrium $e_i$ for $i=1,\dots,n$. And the heteroclinic cycle cannot be asymptotically stable.
 
Following \cite[Theorem 2.9]{KM95}, the heteroclinic cycle $e_1\rightarrow e_3\rightarrow e_{2}\rightarrow e_1$ is stable if the contracting eigenvalue, $c-a$, is stronger than the expanding eigenvalue, $-a$. 
So the heteroclinic cycle is stable if $n=3$ and $$|c-a|>-a\Leftrightarrow c<2a<0.$$ 
\end{proof}

For $b^2+4(c+1)\geq 0$, the unstable manifold of the equilibria $\pm e_i$ is not contained in the heteroclinic cycle.
So the heteroclinic cycle is not asymptotically stable. 
Next, we will focus on the case where there is a heteroclinic network $b^2+4(c+1)<0$.

\section{Dynamic on the invariant sphere}\label{sec:perturb}

Fix $n=3$, $N$ the $3$-ring network, $a=-1$ and $b,c\in\mathbb{R}$ such that $b^2+4(c+1)<0$. 
In this section, we will study the dynamics of the $3$-ring-coupled system on the invariant sphere $S^{n-1}$ when the parameter $d$ is perturbed.
Consider the polynomial $q_d(y,z)=-y^3+b y^2 z + c y z^2+d z^3$ for small $d\in\mathbb{R}$ and the corresponding ring-coupled system given by $Q_N(x,d)\in P_N^3$.
Suppose that $Q_N(.,0)$ is contracting as a consequence of the conditions in Proposition~\ref{prop:contrac}.
Since the conditions in Proposition~\ref{prop:contrac} vary continuously with the coefficient $d$, we have that $Q_N(x,d)$ is still contracting for small values of $d$.
This means that the invariant sphere theorem can still be applied and we study the dynamical system on the sphere $u\in S^{2}$
$$\dot{u}=\mathcal{P}_N(u,d).$$

Recall the $\mathbb{Z}_3$ action $Z(u_1,u_2,u_3)=(u_2,u_3,u_1)$ and the antipodal action $S(u_1,u_2,u_3)=-(u_1,u_2,u_3)$ in $\mathbb{R}^3$.
The previous system is $\mathbb{Z}_3\times \mathbb{Z}_2$ equivariant.
Thus, we can study the dynamics around one of the poles, say $e_1$, and use the symmetry to understand the dynamics around the other poles.
In fact, we can study the dynamics in the fundamental domain, which is the quotient of the sphere by the action of the group $\mathbb{Z}_3\times \mathbb{Z}_2$.

The heteroclinic network connecting the poles $\pm e_1,\pm e_2,\pm e_3$ in the sphere corresponds to two homoclinic connections to a single equilibrium in the fundamental domain when $d=0$.
It follows from Proposition~\ref{prop:hetcycle} that there is a heteroclinic network connecting the poles $\pm e_1,\pm e_2,\pm e_3$ for $d=0$.
The poles belong to the same action orbit.
Denote by $e\equiv \pm e_i $ the corresponding point in the fundamental domain.
The heteroclinic connections between poles with the same signal $\pm e_i \mapsto \pm e_i$ belong to one action orbit and lead to a homoclinic orbit around $e$.
The other heteroclinic connections $\pm e_i \mapsto \mp e_i$ are in another class of the fundamental network and correspond to another homoclinic orbit around $e$.
In the fundamental domain, the heteroclinic network is transformed into two homoclinic orbits connecting the equilibrium $e$ to itself.

Since $b^2+4(c+1)<0$, we can find a neighbourhood $V$ of the two homoclinic orbits with a single equilibrium for $d$ sufficiently small.
Recall the dynamic in the circle $B_i$ as represented in Figure~\ref{fig:dynonbi5}. 
As we change $d$, there is no steady-state bifurcation in the heteroclinic connections. 
Moreover, there exists a smooth perturbation of the equilibrium $e_i$ that we call $u_i(d)$ where $i=\pm1, \pm 2, \pm 3$.
These equilibrium points belong to the same action orbit: $u_1(d)$, $u_3=Z(u_1(d))$, $u_2(d)=Z(Z(u_1(d)))$, $u_{-1}(d)=S(u_1(d))$, $u_{-3}(d)=S(Z(u_1(d)))$ and $u_{-2}(d)=S(Z(Z(u_1(d))))$.
Denote by $u(d)$ the corresponding equilibrium in the fundamental domain.
There exists a neighbourhood of the heteroclinic network where there is no other equilibrium apart from the $6$ perturbed equilibria, $u_{i}(d)$.
And the corresponding neighbourhood of the two homoclinic orbits in the fundamental domain has a single equilibrium. 
We will restrict our attention to such neighbourhoods.

Let $I_{+}, I_{-}\subset V$ be two cross sections transversal to each of the homoclinic orbits in the fundamental domain.
These cross sections are transversal to the stable manifold of $e_i$, $W^s(e_i)$ and they are still transverse to the stable manifold $W^s(u_i(d))$ for $d$ sufficiently small.
Take $D_{+}\subset I_{+}$ and $D_{-}\subset I_{-}$ in a way that trajectories starting in $D_+$ or $D_-$ cross $I_{+}\cup I_{-}$ after some time and define the Poincaré map $$\Pi_d:D_+\cup D_-\to I_{+}\cup I_{-}.$$
This Poincaré map describes the dynamic in a neighbourhood of the two homoclinic orbits.

We will express this Poincaré map in local coordinates around one equilibrium.
The two cross sections $I_{+}, I_{-}$ in the fundamental domain correspond to $12$ cross sections of the heteroclinic network in $S^2$, two in each of the $6$ equilibria.
Let $I^i_{+}, I^i_{-}\subset S^2$ be the two cross sections transversal to the stable manifold of $u_i(d)$, where $i= \pm 1, \pm 2, \pm 3$. 
The domain of the Poincaré map $\Pi_d$ is also lifted to each of the cross sections, $D^i_{+}\subset I^i_{+}$ and $D^i_{-}\subset I^i_{-}$ where $i=\pm1, \pm 2, \pm 3$.
The trajectories starting in $D^1_+$ or $D^1_-$ cross the outgoing sections $O^1_+:=I^{3}_{+}$ and $O^1 _-:= I^{-3}_{-}$.
So, we have a local Poincaré map around the pole $e_1$, $$\Pi^1_d: D^1_+\cup D^1_-\to O^1_{+}\cup O^1_{-}.$$
Note that the forward trajectories starting in $D^1_+$ or $D^1_-$ are inside the half-sphere containing $e_1$.
So, the Poincaré map $\Pi^1_d$ can be deduced from the local coordinates given by the stereographic projection around $e_1$.

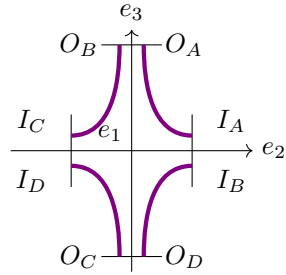
\begin{figure}[ht]
	\centering
\begin{tikzpicture}[scale=0.4, purplecurve/.style={line width=1.6pt,violet}, yellowcurve/.style={line width=1.6pt,green} ] 
\draw (0,0) node[above left] {$e_{1}$}; 
\draw[->] (-4,0) -- (4,0) node[right] {$e_{2}$}; 
\draw[->] (0,-4) -- (0,4) node[above] {$e_{3}$}; 

\draw (2.5,1) node[right] {$I_A$}; 
\draw (2.5,-1) node[right] {$I_B$};
\draw (-2.5,1) node[left] {$I_C$}; 
\draw (-2.5,-1) node[left] {$I_D$};

\draw (0.8,3.5) node[right] {$O_A$}; 
\draw (-0.8,3.5) node[left] {$O_B$};
\draw (0.8,-3.5) node[right] {$O_D$}; 
\draw (-0.8,-3.5) node[left] {$O_C$};

\draw (-1,3.5) -- (1,3.5); 
\draw (-1,-3.5) -- (1,-3.5); 
\draw (-2.0,-1.2) -- (-2.0,1.2); 
\draw (2.0,-1.2) -- (2.0,1.2); 
\draw[purplecurve] (2,0.5) to[out=180,in=270] (0.4,3.5); 
\draw[purplecurve] (2,-0.5) to[out=180,in=90] (0.4,-3.5); 
\draw[purplecurve] (-2,0.5) to[out=0,in=-90] (-0.4,3.5); 
\draw[purplecurve] (-2,-0.5) to[out=0,in=90] (-0.4,-3.5); 
\end{tikzpicture}
\caption{Incoming and outgoing cross sections on the local coordinates around $e_1$.}
	\label{fig:poincaremap}
\end{figure}

Before we deduce the Poincaré map $\Pi^1_d$, we make two observations.
As the stereographic projection sends the half-sphere $\{u_1>0:u\in S^2\}$ into the full plane, the heteroclinic connections are no longer compact.
And the cross sections may be far away from the origin. 
See Figure~\ref{fig:poincaremap}.
The cross sections $D^1_{+}$ and $ D^1_{-}$ are each split into two cross sections by the stable manifold. 
Denote by $I_A$, $I_B$, $I_C$ and $I_D$ the $4$ cross sections as displayed in Figure~\ref{fig:poincaremap}. 
Defining the corresponding outgoing cross sections using the $\mathbb{Z}_3\times \mathbb{Z}_2$ symmetry, we have that $O^1_+=I^{3}_{+}=Z(I^{1}_{+})$ and $O^1 _-= I^{-3}_{-}= S(Z(I^{1}_{-}))$.
So, we obtain $O_A$, $O_B$, $O_C$ and $O_D$ as labelled in Figure~\ref{fig:poincaremap}.
Note that $O^1_+$ and $I^{1}_{+}$ are glued preserving the orientation, and the $O^1_-$ and $I^{1}_{-}$ are glued reversing the orientation.
Thus, in the fundamental domain, we are studying the dynamic on an annulus glued with a Möbius strip. 

The Poincaré map $\Pi_d$ is defined by defining the Poincaré map $\Pi^A_d$, $\Pi^B_d$, $\Pi^C_d$, and $\Pi^D_d$ for each cross section. Identifying the incoming and outgoing cross sections with the same labels, we have that 
$$\Pi^A_d: I_A\to I_A\cup I_B,\quad \Pi^B_d: I_B\to I_C\cup I_D,\quad \Pi^C_d: I_C\to I_A\cup I_B,\quad \Pi^D_d: I_D\to I_C\cup I_D.$$

\subsection{Poincaré maps around the heteroclinic network}

First, we compute the system in local coordinates around the pole $e_1$ on half of the sphere.
 
\begin{lem}\label{lem:2dsystem}
 Let $q(y,z)=a y^3+b y^2 z+c y z^2+d z^3$.	
	The system (\ref{eq:edosphere}) in the half-sphere $\{u\in S^{n-1}:u_1>0 \}$ around the point $e_1\in S^{n-1}$ is equivalent to the following $n-1$ dimensional system:
$$\dot{p}_i=\dfrac{q(p_i,p_{i-1})-p_i q(1,p_n)}{1+p_2^2+\dots +p_n^2},$$
where $p_1=1$, $p_i\in\mathbb{R}$ and $i=2,\dots,n$.

For $n=3$, the Taylor expansion of the previous system around $(p_2,p_3)=(0,0)$ is given by 
$$\begin{cases}
\dot{p}_2 =d+(c-a)p_2+(b-d)p_2^2  -b p_3 p_2-d p_3^2+(2a-c) p_2^3 +(a-2c)p_3^2 p_2 \\
\dot{p}_3 = -a p_3- bp_3^2+a p_3 p_2^2+(2a-c)p_3^3+bp_3^2p_2+cp_3p_2^2+dp_2^3 
\end{cases}+\mathcal{O}(4),$$
where $\mathcal{O}(4)$ denotes the terms of order four or higher in $p_2$ and $p_3$.
\end{lem}

\begin{proof}

We consider the following stereographic projection from the tangent space $T_{e_1}S^{n-1}$ to the half-sphere $\{u\in S^{n-1}:u_1>0 \}$.
Let $P_{e_1}=\{(0,p_2,\dots,p_n)\in \mathbb{R}^n:p_i\in\mathbb{R}\}=T_{e_1}S^{n-1}\cong \mathbb{R}^{n-1}$ and $\phi_{1}:\mathbb{R}^{n-1}\rightarrow S^{n-1}$ given by
$$ (p_2,\dots,p_n)\mapsto \dfrac{(1,p_2,\dots,p_n)}{\|(1,p_2,\dots,p_n)\|}.$$
Note that $\phi_{1}(\mathbb{R}^{n-1})=\{u\in S^{n-1}:u_1>0 \}\subset S^{n-1}$ and 
$$\phi^{-1}_{e_1}(u)=(\dfrac{u_2}{u_1},\dfrac{u_3}{u_1},\dots, \dfrac{u_n}{u_1}).$$
This change of coordinates is invertible and smooth.

Around the equilibrium $e_1$, the system (\ref{eq:edosphere}) can be written in terms of $p=\phi^{-1}_{e_1}(u)\in \mathbb{R}^{n-1}$ as
$$\dot{p}=  D \phi^{-1}_{e_1}(\phi(p)).\dot{u}=D \phi^{-1}_{e_1}(\phi(p)).(Q_N(\phi(p))-\langle Q_N(\phi(p)),\phi(p)\rangle \phi(p))$$
Note that 
$$D \phi^{-1}_{e_1}(u)=\begin{bmatrix}
-\dfrac{u_2}{u_1^2}& \dfrac{1}{u_1}&0&\dots &0\\
-\dfrac{u_3}{u_1^2}& 0&\dfrac{1}{u_1}&\dots &0\\
\vdots& \vdots& \vdots& \ddots& \vdots\\
-\dfrac{u_n}{u_1^2}& 0&0&\dots &\dfrac{1}{u_1}\\
\end{bmatrix}\quad \textrm{ and }\quad D\phi^{-1}_{e_1}(u) u=0.$$
and 
$$\dot{p}= \dfrac{\|(1,p_2,\dots, p_n)\|}{\|(1,p_2,\dots, p_n)\|^3}\begin{bmatrix}
-p_2& 1 &0&\dots &0\\
-p_3& 0&1&\dots &0\\
\vdots& \vdots& \vdots& \ddots& \vdots\\
-p_n & 0&0&\dots &1
\end{bmatrix}Q_N(1,p_2,\dots,p_3)$$
$$=
\dfrac{1}{\|(1,p_2,\dots, p_n)\|^2}\begin{bmatrix}
-p_2& 1 &0&\dots &0\\
-p_3& 0&1&\dots &0\\
\vdots& \vdots& \vdots& \ddots& \vdots\\
-p_n & 0&0&\dots &1
\end{bmatrix}\begin{bmatrix}
q(1,p_n)\\
q(p_2,1)\\
\vdots\\
q(p_n,p_{n-1})
\end{bmatrix}
$$
$$
=\dfrac{1}{1+p_2^2+\dots +p_n^2}\begin{bmatrix}
q(p_2,1)-p_2q(1,p_n)\\
q(p_3,p_2)-p_3q(1,p_n)\\
\vdots\\
q(p_n,p_{n-1})-p_n q(1,p_n)\\
\end{bmatrix}.$$
So $$\dot{p}_i=\dfrac{q(p_i,p_{i-1})-p_i q(1,p_n)}{1+p_2^2+\dots +p_n^2},$$
where $p_1=1$ and $i=2,\dots,n$.

We have the following Taylor expansion up to order $4$ of 
$$\dfrac{1}{(1+p_2^2+\dots + p_n^2)}=1-(p_2^2+\dots + p_n^2)+\mathcal{O}(4)$$
So the vector field can be written as 
$$\dfrac{q(p_i,p_{i-1})-p_i q(1,p_n)}{1+p_2^2+\dots +p_n^2}=q(p_i,p_{i-1})-p_i q(1,p_n)-(p_2^2+\dots +p_n^2)(q(p_i,p_{i-1})-p_i q(1,p_n))+\mathcal{O}(4)$$

For $i=2$, we have 
\begin{align*}
\dot{p}_2=&q(p_2,p_1)-p_2 q(1,p_n)-(p_2^2+\dots +p_n^2)(q(p_2,p_1)-p_2 q(1,p_n))+\mathcal{O}(4)\\
 =&d+(c-a)p_2+(b-d)p_2^2 -d(p_3^2+\dots +p_n^2) -b p_n p_2\\
 & +(2a-c) p_2^3 +(a-c) p_2(p_3^2+\dots +p_{n-1}^2) +(a-2c)p_n^2 p_2 +\mathcal{O}(4)
\end{align*}
For $i=3,\dots,n$, we have 
\begin{align*}
\dot{p}_i	=&q(p_i,p_{i-1})-p_i q(1,p_n)-(p_2^2+\dots +p_n^2)(q(p_i,p_{i-1})-p_i q(1,p_n))+\mathcal{O}(4)\\
					=& -ap_i -bp_np_i\\
					&+q(p_i,p_{i-1})+a p_i(p_2^2+\dots +p_{n-1}^2)+(a-c)p_i p_n^2 +\mathcal{O}(4) 
\end{align*}

For $n=3$, the system (\ref{eq:edosphere}) in the half-sphere around the point $(1,0,0)\in S^2$ is equivalent to the following two-dimensional system:
$$\begin{cases}
\dot{p}_2=d+(c-a)p_2+(b-d)p_2^2  -b p_3 p_2-d p_3^2+(2a-c) p_2^3 +(a-2c)p_3^2 p_2 \\
\dot{p}_3	= -a p_3- bp_3^2+a p_3 p_2^2+(2a-c)p_3^3+bp_3^2p_2+cp_3p_2^2+dp_2^3 
\end{cases}+\mathcal{O}(4)
$$

\end{proof}

In an oversimplification, we discard the second and higher-order terms given in  Lemma~\ref{lem:2dsystem} to compute the local Poincaré map $\Pi^A_d$, $\Pi^B_d$, $\Pi^C_d$, and $\Pi^D_d$.

\begin{lem}\label{lem:poincaremap}
Let $\alpha,\epsilon,\delta >0$. Consider $d,\mu,\nu>0$ small enough such that the Poincaré map $\Pi_d$ of the linear system 
$$
\begin{cases}
\dot{p}_2=\alpha d- \alpha p_2\\
\dot{p}_3=p_3
\end{cases}
$$
with the following incoming cross sections 
$$I_A=\{\epsilon\}\times ]0,\nu[,\quad I_B=\{\epsilon\}\times ]-\nu,0[,\quad
I_C=\{-\epsilon\}\times ]0,\nu[,\quad I_D=\{-\epsilon\}\times ]-\nu,0[,$$
and the following outgoing cross sections 
$$O_A=]0,\mu[\times \{\delta\},\quad O_B=]-\mu,0[ \times \{\delta\}],\quad O_C=]-\mu,0[\times \{-\delta\},\quad O_D=]0,\mu[ \times \{-\delta\},$$
is well defined.
Then, identifying the incoming and outgoing cross sections with the same labels, the Poincaré maps are given by
$$\Pi^{A}_{d}(x) = K_+(d)x^{\alpha}+d,\quad x>0,$$
$$\Pi^{B}_{d}(x) = -(K_+(d)|x|^{\alpha}+d),\quad x<0,$$
$$\Pi^{C}_{d}(x) = K_-(d)x^{\alpha}+d,\quad x>0,$$
$$\Pi^{D}_{d}(x) = -(K_-(d)|x|^{\alpha}+d),\quad x<0,$$
where $K_+(d) = ( \epsilon -d )\delta^{-\alpha}$ and $K_-(d) = -( \epsilon + d )\delta^{-\alpha}$.
\end{lem}

\begin{proof}
Note that the stable manifold of $u_1(d)=(-d,0)$ is given by $\{p_3=0\}$ for any $d$. 
So the orbits starting at the incoming cross sections $I_A$ and $I_C$ will hit the outgoing cross sections $O_A \cup O_B$, for any $d$.
The same is valid for the orbits starting at the incoming cross sections $I_B$ and $I_D$ which will hit the outgoing cross sections $O_C \cup O_D$ for any $d$.

The solutions of the linear system are given by 
$$p_2(t) = k_2 e^{-\alpha t} +d, \qquad p_3(t) = k_3 e^{t},$$
where the constants $k_2,k_3\in\mathbb{R}$ are determined by the initial conditions.

For $\Pi^{A}_{d}$, the trajectories start at the incoming cross section $p_2(0)=\epsilon$ and hit the outgoing cross section $O_A \cup O_B$. 
In this case, the initial conditions are given by $p_2(0) = \epsilon$ and $p_3(0) = x>0$.
Then, $k_2=\epsilon - d$ and $k_3 = x$.
The time $T$ where the trajectory hits the outgoing cross section is determined by the condition $p_3(T) = \delta$. So, $T= -\ln ( \frac{x}{\delta})$. And the Poincaré map is given by 
$$\Pi^{A}_{d}(x) = p_2(T)=K_+(d)x^{\alpha}+d,$$
where $K_+(d) = ( \epsilon -d )\delta^{-\alpha}$.

For $\Pi^{B}_{d}$, the initial conditions are given by $p_2(0) = \epsilon$ and $p_3(0) = x<0$ and the crossing condition is given by $p_3(T) = -\delta$, where $T$ is the hitting time. 
Then, $k_2=\epsilon - d$, $k_3 = x$ and the same hitting time $T= -\ln ( \frac{-x}{\delta})$.
As we identifying $O_C$ with $I_C$ and $O_D$ with $I_D$, the Poincaré map is defined by the  symmetric value of $p_2(T)$, that is,
$$\Pi^{B}_{d}(x) = -p_2(T)=-K_+(d)|x|^{\alpha}-d.$$

For $\Pi^{C}_{d}$, $p_2(0) = -\epsilon$, $p_3(0) = x>0$ and  $p_3(T) = \delta$, we have that 
$k_2=-\epsilon-d$, $k_3=x>0$ and the Poincaré map is given by
$$\Pi^{C}_{d}(x) = p_2(T)=K_-(d)x^{\alpha}+d,$$
where $K_-(d) = -( \epsilon + d )\delta^{-\alpha}$.

The last Poincaré map is also the symmetric value of the respective $p_2(T)$, that is,
$$\Pi^{D}_{d}(x) = -p_2(T)=-K_-(d)|x|^{\alpha}-d,\quad x<0.$$
\end{proof}

\begin{figure}[ht]
	\centering
		\includegraphics[width=0.20\textwidth]{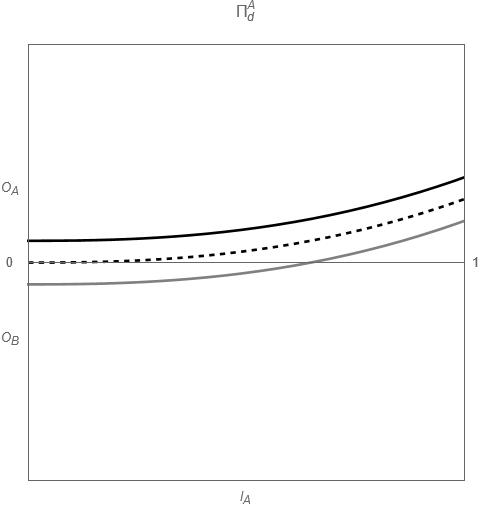}
		\includegraphics[width=0.20\textwidth]{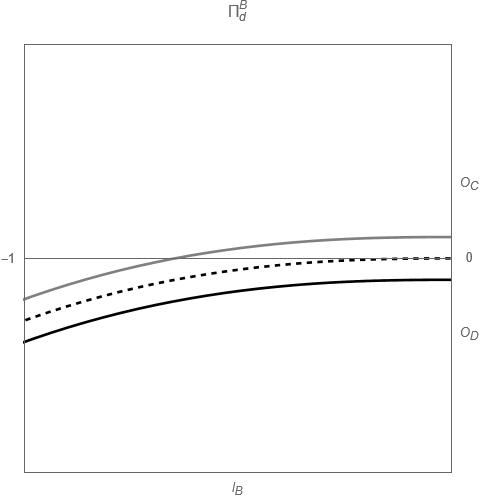}
		\includegraphics[width=0.20\textwidth]{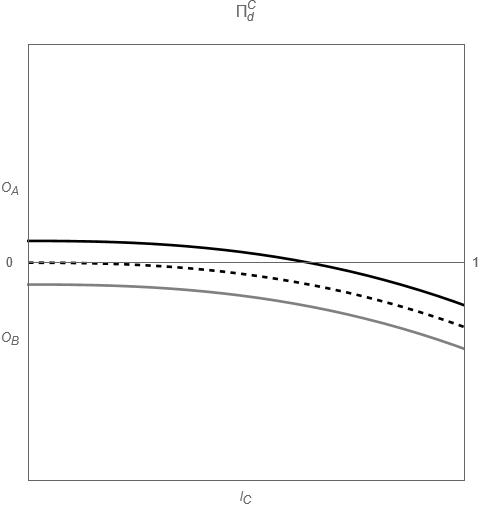}
		\includegraphics[width=0.20\textwidth]{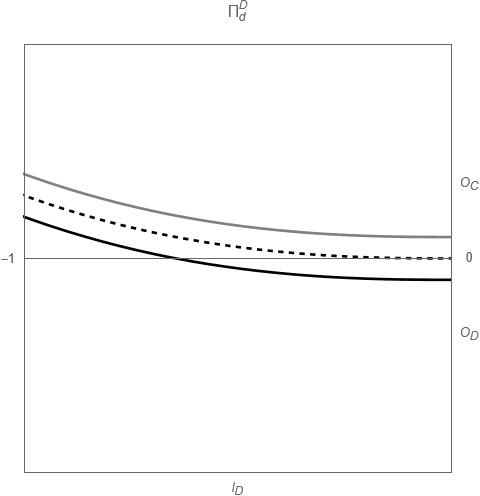}
		\caption{Plot of the Poincaré maps for $\alpha>1$ and $d>0$(thick), $d=0$(dashed), $d<0$(grey).}	
	\label{fig:piadneg}
\end{figure}

The Poincaré maps are plotted in Figure~\ref{fig:piadneg} for different values of $d$ and $\alpha>1$.

Note that $K_+(0)=\epsilon \delta^{-\alpha}>0$ and $K_-(0)=-\epsilon \delta^{-\alpha}<0$, we assume that $K_+(d)>0$ and $K_-(d)<0$ by keeping $d$ small enough. As we identify the outgoing section with the incoming section, we take $\delta=1/\epsilon$, since 
$$Z(\phi_{1}(\epsilon, x))=Z(\dfrac{(1,\epsilon, x)}{\|(1,\epsilon, x)\|})=\dfrac{(\epsilon, x, 1)}{\|(\epsilon, x, 1)\|}=\dfrac{(1, x/\epsilon, 1/\epsilon)}{\|(1, x/\epsilon, 1/\epsilon)\|}=\phi_{1}\left(\dfrac{x}{\epsilon}, \dfrac{1}{\epsilon}\right).$$
So we can assume that $|K_+(0)|=|K_-(0)|=\epsilon^{1+\alpha}<1$ by taking $\epsilon$ sufficiently small.

\subsection{Breaking the heteroclinic connections}

There are four different scenarios depending on the sign of $d$ and whether the heteroclinic network is stable or not.

\begin{theo}\label{theo:periodic_orbits}
	Let $n=3$, $\lambda>0$ and $b,c, d\in\mathbb{R}$ such that $b^2+4(c+1)<0$ and $d$ is sufficiently small.
	The system (\ref{eq:edoring}) on the invariant sphere $S^2$ has a heteroclinic network for $d=0$ and define $\alpha=-(c+1)>0$.
Assume that the system (\ref{eq:edosphere}) is topologically equivalent to the linear system in Lemma~\ref{lem:poincaremap} in a neighbourhood of the heteroclinic connections around the equilibrium $e_1$.
Then, we have the following dynamics around the destroyed heteroclinic network:
\begin{itemize}
\item If $\alpha>1$ and $d>0$, then there are three attracting periodic orbits passing nearby the poles $e_1, e_3,e_2, e_1$; $-e_1, -e_3,-e_2, -e_1$; and $e_1, -e_3, e_2, -e_1, e_3, -e_2, e_1$. 
\item If $\alpha>1$ and $d<0$, then there is one attracting periodic orbit passing nearby the poles $e_1, -e_3, -e_2, e_1,e_3, -e_2, -e_1, e_3, e_2, -e_1, -e_3, e_2, e_1$.
\item If $\alpha<1$ and $d>0$, then there is one repelling periodic orbit passing nearby the poles $e_1, -e_3, e_2, -e_1, e_3, -e_2, e_1$. 
\item If $\alpha<1$ and $d<0$, then there are two repelling periodic orbits passing nearby the poles $e_1, e_3,e_2,e_1$ and $-e_1, -e_3,-e_2,-e_1$. 
\end{itemize}
\end{theo}

\begin{proof}

It follows from the invariant sphere theorem that the system (\ref{eq:edoring}) is topologically equivalent to the system (\ref{eq:edosphere}) on the attracting invariant sphere $S^2$. 
So, the proof follows from a carefully study of the possible compositions of the Poincaré maps $\Pi^{A}_{d}$, $\Pi^{B}_{d}$, $\Pi^{C}_{d}$ and $\Pi^{D}_{d}$ in the $4$  different cases. 

First, we fix $\alpha>1$ and $d>0$.
See Figure~\ref{fig:3_periodic_orbit_}.
Note that  $\Pi^{A}_{d}(x)>0$ for every $x>0$. 
So the trajectories starting in $I_A$ always cross $O_A\approx I_A$.
Moreover, since $0<K_+(d)<1$ and $\alpha>1$, the function $\Pi^{A}_{d}$ is a contraction map. 
Thus, there is a unique attracting fixed point of $\Pi^{A}_{d}$ which corresponds, using the antipodal symmetry, to two attracting periodic orbits of the system (\ref{eq:edosphere}) close to the destroyed heteroclinic network.
Since $\Pi^{B}_{d}(x)<0$, the trajectories starting in $I_B$ always cross $O_D\approx I_D$.
The trajectories starting in $I_C$ cross $O_A\approx I_A$ since $\Pi^{C}_{d}(x)>0$ for every $x>0$ sufficiently small. 
For the last section $I_D$, the trajectories cross $O_D\approx I_D$ since $\Pi^{D}_{d}(x)<0$ for every $x<0$ sufficiently small.
Furthermore, the map $\Pi^{D}_{d}$ is a contraction map since $-1<K_-(d)<0$ and $\alpha>1$.
Thus, there is a unique attracting fixed point of $\Pi^{D}_{d}$.
(In the Möbius strip, this periodic orbit makes a single loop around the strip, however the orbits approaching it do it from both sides.)
The corresponding periodic orbit  must be invariant by the antipodal symmetry, otherwise there would be two intersection periodic orbits, which is absurd. 
So there are two periodic orbits in the fundamental domain as plotted on the left side of Figure~\ref{fig:3_periodic_orbit_}.
These periodic orbits correspond to attracting periodic orbits of the system (\ref{eq:edosphere}) close to the destroyed heteroclinic network.
These periodic orbits are similar to the ones plotted on the right side of Figure~\ref{fig:3_periodic_orbit_}.

The orbits starting in $I_B$ close to the heteroclinic network are attracted by this periodic orbit, because $\Pi^{D}_{d}(d)<|d|+d<0$.
In short, the invariant sphere has three attracting periodic orbits close to the destroyed heteroclinic network for $\alpha>1$ and $d<0$. They follow the following points: $e_1, e_3,e_2, e_1$; $-e_1, -e_3,-e_2, -e_1$; and $e_1, -e_3, e_2, -e_1, e_3, -e_2, e_1$.

\begin{figure}[ht]
	\centering
\begin{tikzpicture}[scale=0.4, purplecurve/.style={line width=1.6pt,blue}, yellowcurve/.style={line width=1.6pt,green} ] 
\draw (0,0) node[above left] {$e_{1}$}; 
\draw[->] (-4,0) -- (4,0) node[right] {$e_{2}$}; 
\draw[->] (0,-4) -- (0,4) node[above] {$e_{3}$}; 

\draw (3.5,1.5) node {$A$}; 
\draw (3.5,-1.5) node {$B$};
\draw (-3.5,1.5) node {$C$}; 
\draw (-3.5,-1.5) node {$D$};

\draw (0.8,3.5) node[right] {$A$}; 
\draw (-0.8,3.5) node[left] {$B$};
\draw (0.8,-3.5) node[right] {$D$}; 
\draw (-0.8,-3.5) node[left] {$C$};

\draw (-1,3) -- (1,3); 
\draw (-1,-3) -- (1,-3); 
\draw (-3.0,-1) -- (-3.0,1); 
\draw (3.0,-1) -- (3.0,1); 
\draw[dashed] (0.5,-3.0) -- (0.5,3);
\draw[yellowcurve] (3,0.7) to[out=180,in=270] (0.7,3); 
\draw[purplecurve] (-3,-0.3) to[out=0,in=90] (0.3,-3); 
\end{tikzpicture}
		\includegraphics[width=0.30\textwidth]{"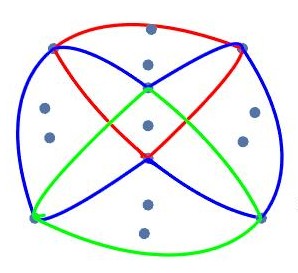"}
	\caption{Dynamics around the destroyed homoclinic/heteroclinic network for $\alpha>1$ and $d>0$. Left: periodic orbits in the fundamental domain. Right: three periodic orbits in the invariant sphere ($a=-1$, $b=1/4$, $c=-3$, $d=0.001$ and $\lambda=1$).}	
	\label{fig:3_periodic_orbit_}
\end{figure}
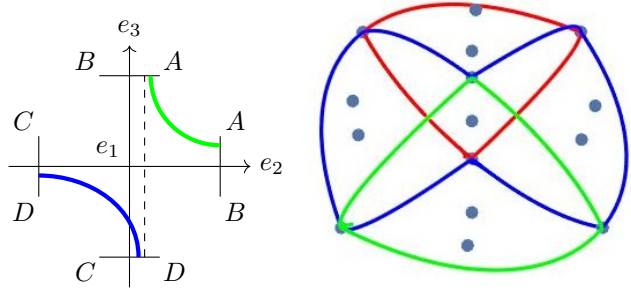

Next, we consider $\alpha>1$ and $d<0$.
See Figure~\ref{fig:periodic_orbit_alpha_gt_1_d_neg}.
Note that $\Pi^{A}_{d}(x)<0$ for every $x>0$ small enough. 
Thus, every trajectory starting in $I_A$ close to the heteroclinic network crosses $O_B\approx I_B$.
The trajectories starting in $I_D$ cross $O_C\approx I_C$ since $\Pi^{D}_{d}(x)>-d>0$ where $x<0$.
For the section $I_C$, we have that $\Pi^{C}_{d}(x)<d<0$ where $x>0$.
So the trajectories starting in $I_C$ cross $O_B\approx I_B$.
Now, we study the dynamic on the section $I_B$.
Note that $\Pi^{B}_{d}(x)>0$ if and only if $0>x>y^*$, where $y^*=-(-d/K_+(d))^{1/\alpha}$. 
And the trajectories starting in $I_B$ close to the heteroclinic connection cross the section $O_C\approx I_C$ which returns to the cross section $O_B\approx I_B$.
We check that the composition $\Pi^{C}\circ \Pi^{B}_{d}$ is invariant by the interval $]y^*,0[$.
We have that
$$x\in ]y^*,0[ \Rightarrow \Pi^{B}_{d}(x)\in ]0,-d[ \Rightarrow \Pi^{C}_{d}(\Pi^{B}_{d}(x))\in ]\Pi^{C}_{d}(-d),d[.$$
We need to check that $\Pi^{C}_{d}(-d)>y^*$ for small values of $d<0$.

Taking $\delta=1/\epsilon$ and $\tilde{d}=-d$, we have that $\Pi^{C}_{d}(-d)>y^*$ is equivalent to 
$$ \left(-\frac{d}{\epsilon^{\alpha+1}-\epsilon^{\alpha} d}\right)^{1/\alpha}>(\epsilon^{\alpha+1}+\epsilon^{\alpha} d)(-d)^\alpha  - d
\Leftrightarrow 
\left(\frac{\tilde{d}}{\epsilon^{\alpha+1}+\epsilon^{\alpha} \tilde{d}}\right)^{1/\alpha}>(\epsilon-\tilde{d})(\epsilon\tilde{d})^\alpha +\tilde{d}
$$
$$\Leftrightarrow 
\tilde{d}-\epsilon^{\alpha}(\epsilon+\tilde{d})((\epsilon- \tilde{d})(\epsilon\tilde{d})^\alpha +\tilde{d})^{\alpha}>0\Leftrightarrow 
h(\tilde{d})>0.
$$
for small values of $\tilde{d}>0$, where 
$h(\tilde{d})=\tilde{d}-\epsilon^{\alpha}(\epsilon+ \tilde{d})((\epsilon- \tilde{d})(\epsilon\tilde{d})^\alpha +\tilde{d})^{\alpha}$.
Note that $h(0)=0$ and
$$h'(\tilde{d})=1-\epsilon^{\alpha}((\epsilon- \tilde{d})(\epsilon\tilde{d})^\alpha +\tilde{d})^{\alpha}-\alpha\epsilon^{\alpha}(\epsilon+ \tilde{d})((\epsilon- \tilde{d})(\epsilon\tilde{d})^\alpha +\tilde{d})^{\alpha-1}(1-(\epsilon\tilde{d})^\alpha +\alpha \epsilon(\epsilon- \tilde{d})(\epsilon\tilde{d})^{\alpha-1})>0.$$
 for small values of $\tilde{d}>0$, since $\alpha>1$. 
Then $h(\tilde{d})>0$, for small values of $\tilde{d}>0$ which means that $\Pi^{C}_{d}(-d)>y^*$ for small values of $d<0$.
 Moreover, since $0<K_+(d)<1$ and $\alpha>1$, the function $\Pi^{C}\circ \Pi^{B}_{d}$ is a contraction map in this interval.
Thus, the composition $\Pi^{C}\circ \Pi^{B}_{d}$ has an attracting fixed point which corresponds to a periodic orbit in the fundamental domain, as displayed in the left side of Figure~\ref{fig:periodic_orbit_alpha_gt_1_d_neg}.
This periodic orbit is lifted to the sphere as an attracting periodic orbit of the system (\ref{eq:edosphere}) close to the destroyed heteroclinic network as displayed in the Center of Figure~\ref{fig:periodic_orbit_alpha_gt_1_d_neg}.
This periodic orbit is illustrated in Figure~\ref{fig:periodic_orbit_alpha_gt_1_d_neg}.
And it passes close to the ``poles" in the following order: $e_1, -e_3, -e_2, e_1,e_3, -e_2, -e_1, e_3, e_2, -e_1, -e_3, e_2, e_1$. 
The periodic orbit passes closes to each equilibria twice before completes its period.
See the right side of Figure~\ref{fig:periodic_orbit_alpha_gt_1_d_neg} for a zoom around one equilibrium point of the destroyed heteroclinic network.
Moreover, the trajectories starting in $I_A$ and $I_D$ close to the heteroclinic network are also attracted by this periodic orbit, since $\Pi^{A}_{d}(x)>d>\Pi^{C}_{d}(-d)>y^*$ and $\Pi^{D}_d(x)>-d \Rightarrow \Pi_d^{C}(\Pi^{D}_d(x))>\Pi^{C}_{d}(-d)>y^*$.

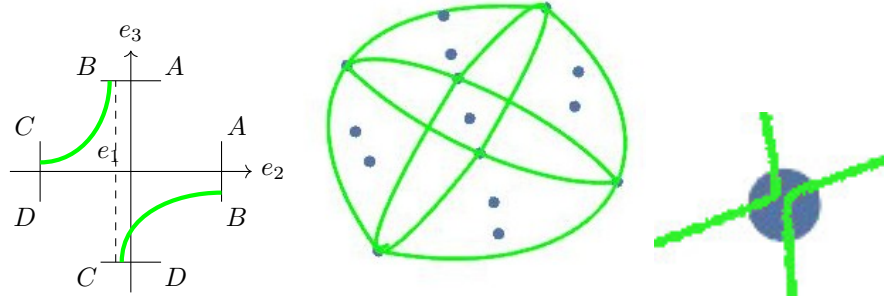
\begin{figure}[ht]
	\centering
	\begin{tikzpicture}[scale=0.4, purplecurve/.style={line width=1.6pt,green}, yellowcurve/.style={line width=1.6pt,yellow!80!orange} ] 
\draw (0,0) node[above left] {$e_{1}$}; 
\draw[->] (-4,0) -- (4,0) node[right] {$e_{2}$}; 
\draw[->] (0,-4) -- (0,4) node[above] {$e_{3}$}; 

\draw (3.5,1.5) node {$A$}; 
\draw (3.5,-1.5) node {$B$};
\draw (-3.5,1.5) node {$C$}; 
\draw (-3.5,-1.5) node {$D$};

\draw (0.8,3.5) node[right] {$A$}; 
\draw (-0.8,3.5) node[left] {$B$};
\draw (0.8,-3.5) node[right] {$D$}; 
\draw (-0.8,-3.5) node[left] {$C$};
\draw (-1,3) -- (1,3); 
\draw (-1,-3) -- (1,-3); 
\draw (-3.0,-1) -- (-3.0,1); 
\draw (3.0,-1) -- (3.0,1); 
\draw[dashed] (-0.5,-3.0) -- (-0.5,3);
\draw[purplecurve] (-3, 0.3) to [out=0,in=270] (-0.7,3); 
\draw[purplecurve] (3,-0.7) to[out=180,in=90] (-0.3,-3); 
\end{tikzpicture}
		\includegraphics[width=0.30\textwidth]{"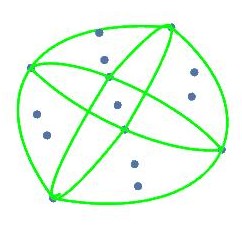"}
		\includegraphics[width=0.20\textwidth]{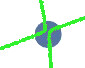}
	\caption{Dynamics around the destroyed homoclinic/heteroclinic network for $\alpha>1$ and $d<0$. Left: peridoic orbit in the fundamental domain. Center: periodic orbit in the invariant sphere ($a=-1$, $b=1/4$, $c=-3$, $d=-0.001$ and $\lambda=1$). Right: Zoom around one equilibrium point of the destroyed heteroclinic network.}	
	\label{fig:periodic_orbit_alpha_gt_1_d_neg}
\end{figure}

Now, we consider $\alpha<1$ and $d>0$.
Note that $\Pi^{A}_{d}(x)>d>0$ and the trajectories starting in $I_A$ cross $O_A\approx I_A$ going away from the connection.
The trajectories starting in $I_D$ cross $O_D\approx I_D$ when $x$ is sufficiently small.
Moreover, there exists a fixed point of $\Pi^{D}$ because $(\Pi^{D})^{-1}(]-y^*,0[)\subset]-y^*,0[$ and every pre image reduces the diameter of the interval, where $y^*_-=(\Pi^{D})^{-1}(0)=-(-d/K_-(d))^{1/\alpha}$.
This fixed point is repelling since $\alpha<1$ and it corresponds to a repelling periodic orbit of the system (\ref{eq:edosphere}) close to the destroyed heteroclinic connection. 
And it follows the following order: $e_1, -e_3, e_2, -e_1, e_3, -e_2, e_1$.

The trajectories that start in $I_C$ close to the network cross $O_A\approx I_A$ and go away from the destroyed connection, $\Pi^{C}(x)>0$ for $x$ sufficiently small.
Finally, for the section $I_B$, the trajectories cross $O_D\approx I_D$ when $x$ is sufficiently small, because $\Pi^{B}(x)<-d$.
So when the trajectories hit the cross section $O_D$, they are not arbitrary close to the destroyed connection.
If the trajectory is still in a neighbourhood of the destroyed heteroclinic network, then we apply the map $\Pi^{D}_{d}$.
As $\Pi^{D}_{d}(-d)>0$ for $d$ sufficiently small, the trajectories will then cross the outgoing connection $O_C\approx I_C$.
And they will leave the neighbourhood of the homoclinic orbit via $I_A$.

\begin{figure}
	\centering
\begin{tikzpicture}[scale=0.4, purplecurve/.style={line width=1.6pt,violet}, yellowcurve/.style={line width=1.6pt,yellow!80!orange} ] 
\draw (0,0) node[above right] {$e_{1}$}; 
\draw[->] (-4,0) -- (4,0) node[right] {$e_{2}$}; 
\draw[->] (0,-4) -- (0,4) node[above] {$e_{3}$}; 
\draw (3.5,1.5) node {$A$}; 
\draw (3.5,-1.5) node {$B$};
\draw (-3.5,1.5) node {$C$}; 
\draw (-3.5,-1.5) node {$D$};

\draw (0.8,3.5) node[right] {$A$}; 
\draw (-0.8,3.5) node[left] {$B$};
\draw (0.8,-3.5) node[right] {$D$}; 
\draw (-0.8,-3.5) node[left] {$C$};

\draw (-1,3) -- (1,3); 
\draw (-1,-3) -- (1,-3); 
\draw (-3.0,-1) -- (-3.0,1); 
\draw (3.0,-1) -- (3.0,1); 
\draw[dashed] (0.5,-3.0) -- (0.5,3);
\draw[yellowcurve] (-3,-0.3) to[out=0,in=90] (0.3,-3); 
\end{tikzpicture}
	\begin{tikzpicture}[scale=0.4, purplecurve/.style={line width=1.6pt,violet}, yellowcurve/.style={line width=1.6pt,yellow!80!orange} ] 
\draw (0,0) node[above right] {$e_{1}$}; 
\draw[->] (-4,0) -- (4,0) node[right] {$e_{2}$}; 
\draw[->] (0,-4) -- (0,4) node[above] {$e_{3}$}; 
\draw (3.5,1.5) node {$A$}; 
\draw (3.5,-1.5) node {$B$};
\draw (-3.5,1.5) node {$C$}; 
\draw (-3.5,-1.5) node {$D$};

\draw (0.8,3.5) node[right] {$A$}; 
\draw (-0.8,3.5) node[left] {$B$};
\draw (0.8,-3.5) node[right] {$D$}; 
\draw (-0.8,-3.5) node[left] {$C$};

\draw (-1,3) -- (1,3); 
\draw (-1,-3) -- (1,-3); 
\draw (-3.0,-1) -- (-3.0,1); 
\draw (3.0,-1) -- (3.0,1); 
\draw[dashed] (-0.5,-3.0) -- (-0.5,3);
\draw[yellowcurve] (3,0.3) to[out=180,in=-90] (0.3,3); 
\end{tikzpicture}
\caption{For $\alpha<1$, there is one fixed point of the local Poincaré maps. Left side displays the case $d>0$, and the right corresponds to $d<0$.}	
	\label{fig:periodic_orbit_alpha_ls_1_d_neg}
\end{figure}
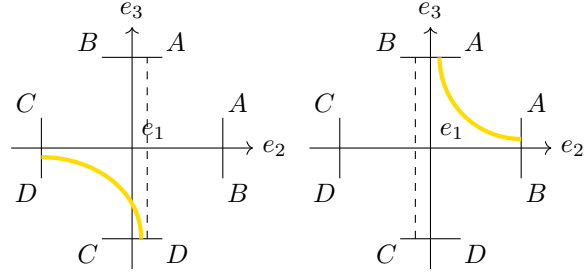

Last, take $\alpha<1$ and $d<0$.
For trajectories starting in $I_A$, we have that $\Pi^{A}_{d}(x)<0$ for every $x>0$ sufficiently small and they cross the section $O_B\approx I_B$.
However, $x$ must be exceptional close to zero to $\Pi^{A}_{d}(x)$ be negative.
If $x>(-d/K_+(d))^{1/\alpha}$, the trajectory goes back to $O_A$. 
Since $\alpha<1$, then $(-d/K_+(d))^{1/\alpha}$ is sufficiently small to be in the neighbourhood of the destroyed connection.
And some of the trajectories starting in $I_A$ return to $O_A\approx I_A$ again.
For $d$ small enough, we have that 
$$](-d/K_+(d))^{1/\alpha}, (-2d/K_+(d))^{1/\alpha}[\subset \Pi^{A}_{d}(](-d/K_+(d))^{1/\alpha}, (-2d/K_+(d))^{1/\alpha}[)=]0,-d[.$$
Since $\Pi^{A}_{d}$ is expanding, the successive preimages of this interval are getting closer to the fixed point of $\Pi^{A}_{d}$.
This fixed point is repelling since $\alpha<1$.
Thus, there are two repelling periodic orbits close to the destroyed heteroclinic network that follow the equilibrium: $e_1, e_3,e_2,e_1$ and $-e_1, -e_3,-e_2,-e_1$.
\end{proof}

We finish with the following remark about the arc length of the merging periodic orbits when $d\to 0$. 
We define the arc length of the periodic orbit by the number of times that the trajectory passes close to the poles.
For example, the periodic orbit given in Theorem~\ref{theo:periodic_orbits} that passes close to the poles in the following order: $e_1, -e_3, -e_2, e_1,e_3, -e_2, -e_1, e_3, e_2, -e_1, -e_3, e_2, e_1$ has arc length equal to 12. 
We say that the heteroclinic network has arc length equal to the number of heteroclinic connections that it has, which are $12$ for $n=3$.

\begin{rem}
For $d$ sufficiently small, the periodic orbits given in Theorem~\ref{theo:periodic_orbits} are close to the destroyed heteroclinic network.

If $\alpha>1$, the total arc length of the attracting periodic orbits converges to the arc length of the heteroclinic network, $12$, when $d\to 0$.
\begin{itemize}
\item If $\alpha>1$ and $d>0$, then the three periodic orbits have arc length $3$, $3$, and $6$. 
\item If $\alpha>1$ and $d<0$, then the arc length of the periodic orbit converges to $12$.
\end{itemize}

If $\alpha<1$, the total arc length of the periodic orbits converges to half the arc length of the heteroclinic network $6$, when $d\to 0$.
\begin{itemize}
\item If $\alpha<1$ and $d>0$, then the one repelling periodic orbit has arc length $6$.
\item If $\alpha<1$ and $d<0$, then each of the two repelling periodic orbits has arc length $3$.
\end{itemize}
\end{rem}

\bibliographystyle{plain}

\end{document}